\documentclass[a4paper,10pt]{article}

\newtheorem{theorem}{Theorem}[section]
\newtheorem{lemma}[theorem]{Lemma}

\newtheorem{corollary}[theorem]{Corollary}

\newenvironment{proof}[1][Proof]{\begin{trivlist}
\item[\hskip \labelsep {\bfseries #1}]}{\end{trivlist}}

\newenvironment{example}[1][Example]{\begin{trivlist}
\item[\hskip \labelsep {\bfseries #1}]}{\end{trivlist}}

\newcommand{\qed}{\nobreak \ifvmode \relax \else      \ifdim\lastskip<1.5em \hskip-\lastskip
\hskip1.5em plus0em minus0.5em \fi \nobreak
\vrule height0.75em width0.5em depth0.25em\fi}

\usepackage{amssymb}
\usepackage{mathrsfs}

 \title{Differences of Augmented Staircase Skew Schur Functions}
 \author{Matthew Morin,}

\begin{document}

\begin{center}

{\huge Differences of Augmented Staircase Skew Schur Functions}

\

{\LARGE Matthew Morin}

University of British Columbia

mjmorin@math.ubc.ca
\

\end{center}



\

\begin{center}
\textbf{Abstract}
\end{center}

We define a fat staircase to be a Ferrers diagram corresponding to a partition of the form $(n^{\alpha_n}, {n-1}^{\alpha_{n-1}},\ldots, 1^{\alpha_1})$, where $\alpha = (\alpha_1,\ldots,\alpha_n)$ is a composition, or the $180^\circ$ rotation of such a diagram. 
We look at collections of skew diagrams consisting of a fixed fat staircase augmented with all hooks of a given size. 
Among these diagrams we determine precisely which pairs give a Schur-positive difference. 
We extend this classification to collections of fat staircases augmented with hook-complements.

\setlength{\unitlength}{0.1mm}

\newcommand{\lroof}{\mbox{\begin{picture}(3,0)
                                \put(0,0){\line(0,1){10}}

                                \put(0,10){\line(1,0){3}}

                           \end{picture}}}

\newcommand{\rroof}{\mbox{\begin{picture}(3,0)
                                \put(3,0){\line(0,1){10}}

                                \put(3,10){\line(-1,0){3}}

                           \end{picture}}}

\newcommand{\longlroof}{\mbox{\begin{picture}(3,0)
                                \put(0,-3){\line(0,1){10}}

                                \put(0,7){\line(1,0){3}}

                           \end{picture}}}

\newcommand{\longrroof}{\mbox{\begin{picture}(3,0)
                                \put(3,-3){\line(0,1){10}}

                                \put(3,7){\line(-1,0){3}}

                           \end{picture}}}

\newcommand{\hooka}{\mbox{\begin{picture}(0,0)
                                \put(-5,-3){ \framebox(10,10){} }
                                \put(-5,-13){ \framebox(10,10){} }
                                \put(-5,-23){ \framebox(10,10){} }
                                \put(-5,-33){ \framebox(10,10){} }
                                \put(-5,-43){ \framebox(10,10){} }
                                \put(-5,-53){ \framebox(10,10){} }
                            \end{picture}}}

\newcommand{\hookb}{\mbox{\begin{picture}(0,0)
                                \put(-5,-3){ \framebox(10,10){} }
                                \put(5,-3){ \framebox(10,10){} }
                                \put(-5,-13){ \framebox(10,10){} }
                                \put(-5,-23){ \framebox(10,10){} }
                                \put(-5,-33){ \framebox(10,10){} }
                                \put(-5,-43){ \framebox(10,10){} }
                            \end{picture}}}

\newcommand{\hookc}{\mbox{\begin{picture}(0,0)
                                \put(-5,-3){ \framebox(10,10){} }
                                \put(5,-3){ \framebox(10,10){} }
                                \put(15,-3){ \framebox(10,10){} }
                                \put(-5,-13){ \framebox(10,10){} }
                                \put(-5,-23){ \framebox(10,10){} }
                                \put(-5,-33){ \framebox(10,10){} }
                            \end{picture}}}

\newcommand{\hookd}{\mbox{\begin{picture}(0,0)
                                \put(-5,-3){ \framebox(10,10){} }
                                \put(5,-3){ \framebox(10,10){} }
                                \put(15,-3){ \framebox(10,10){} }
                                \put(25,-3){ \framebox(10,10){} }
                                \put(-5,-13){ \framebox(10,10){} }
                                \put(-5,-23){ \framebox(10,10){} }
                            \end{picture}}}

\newcommand{\hooke}{\mbox{\begin{picture}(0,0)
                                \put(-5,-3){ \framebox(10,10){} }
                                \put(5,-3){ \framebox(10,10){} }
                                \put(15,-3){ \framebox(10,10){} }
                                \put(25,-3){ \framebox(10,10){} }
                                \put(35,-3){ \framebox(10,10){} }
                                \put(-5,-13){ \framebox(10,10){} }
                            \end{picture}}}

\newcommand{\hookf}{\mbox{\begin{picture}(0,0)
                                \put(-5,-3){ \framebox(10,10){} }
                                \put(5,-3){ \framebox(10,10){} }
                                \put(15,-3){ \framebox(10,10){} }
                                \put(25,-3){ \framebox(10,10){} }
                                \put(35,-3){ \framebox(10,10){} }
                                \put(45,-3){ \framebox(10,10){} }
                            \end{picture}}}

\section{Introduction}

The Schur functions are perhaps best known as a basis of the ring of symmetric functions. 
As such, the numbers $c_{\mu \nu}^{\lambda}$, commonly known as the \textit{Littlewood-Richardson coefficients}, that arise in the product
\[ s_\mu s_\nu = \sum_{\lambda} c_{\mu \nu}^{\lambda} s_{\lambda}, \]
are of paramount importance to this structure of this ring.

This structure appears in several other areas. 
In the representations of the symmetric group, the 
Specht modules can be placed in one-to-one correspondence with the Schur functions and given two Specht modules $S^\mu$ and $S^\nu$ we have
\[ (S^\mu \otimes S^\nu)\uparrow^{S_n} = \bigoplus_{\lambda} c_{\mu \nu}^{\lambda} S^{\lambda}, \]
and, 
in the cohomology ring of the Grassmannian, the Schubert classes are in correspondence to the Schur functions and the cup product of each pair $\sigma_\mu$, $\sigma_\nu$ of Schubert classes satisfies
\[ \sigma_\mu \cup \sigma_\nu = \sum_{\lambda} c_{\mu \nu}^{\lambda} \sigma_\lambda.\]

It is well known that $c_{\mu \nu}^{\lambda} \geq 0.$ 
Thus each product $s_\mu s_\nu$ gives rise to a linear combination of Schur functions with non-negative coefficients. 
Such an expression is said to be \textit{Schur-positive}.
In recent years, there has been significant interest in determining instances of Schur-positivity in expressions of the form
\[ s_{\mu} s_{\nu} - s_{\lambda} s_{\rho} \textrm{ } \textrm{ and } \textrm{ } s_{\lambda / \mu} - s_{\rho / \nu}.\]
A collection of work in this vein includes \cite{{g2},{complementcite},{lpp},{m},{mvw2}}.  
Each of these Schur-positive differences gives a set of inequalities that the corresponding Littlewood-Richardson coefficients must satisfy. 
In \cite{fomin3}, a Schur-positivity result was used to characterize the eigenvalues of a Hermatian matrix. 
Further, any Schur-positive homogeneous symmetric function of degree $n$ can be expressed as a Frobenius image of some representation of $S_n$.

In this paper we shall define certain types of staircase diagrams and answer the question of Schur-positivity of each difference of any pair of hook augmentations of a given staircase and each difference of any pair of hook complement augmentations of a given staircase.

\section{Preliminaries}

A \textit{partition} $\lambda$\label{def:lambda} of a positive integer $n$, written $\lambda \vdash n$, is a sequence of weakly decreasing positive integers $\lambda = (\lambda_1,\lambda_2, \ldots, \lambda_k)$ with $\sum_{i=1}^{k} \lambda_i = n$. 
We shall use $j^r$ to denote the sequence $j,j,\ldots, j$ consisting of $r$ $j$'s.
Under this notation, $\lambda = (k^{r_k}, {k-1}^{r_{k-1}}, \ldots ,1^{r_1})$ \label{parts} denotes the partition which has $r_1$ parts of size one, $r_2$ parts of size two, \ldots,  and $r_k$ parts of size $k$.

We say $\alpha=(\alpha_1,\alpha_2,\ldots,\alpha_k)$ is a \label{comp} \textit{composition} of $n$ if each $\alpha_i$ is a positive integer and $\sum_{i=1}^{k} \alpha_i = n$.
If we relax this condition to allow each $\alpha_i$ to be non-negative, then we call the result a \textit{weak composition}
If $\lambda$ is either a partition or a composition we call each $\lambda_i$ a \textit{part} of $\lambda$, and if $\lambda$ has exactly $k$ parts we say $\lambda$ is of \textit{length} $k$ and write $l(\lambda)=k$.
The \textit{size} of $\lambda$ is given by $| \lambda |=\sum_{i=1}^{k} \lambda_i$. 

Given a partition $\lambda$, we can represent it via the diagram of left-justified rows of boxes whose $i$-th row contains $\lambda_i$ boxes. 
The diagrams of these type are called \textit{Ferrers diagrams}. 
We shall use the symbol $\lambda$ when refering to both the partition and its Ferrers diagram.

Whenever we find a diagram $\mu$ contained in a diagram $\lambda$ as a subset of boxes, we write $\mu \subseteq \lambda$ and say that $\mu$ is a \textit{subdiagram} of $\lambda$.
In this case we can form the \textit{skew \label{skew} diagram} $\lambda / \mu$ by removing the boxes of $\mu$ from the top-left corner of $\lambda$.

A \textit{hook} is the Ferrers diagram corresponding to a partition $\lambda$ that satisfies $\lambda_i \leq 1$ for all $i>1$. 
Hence a hook has at most one row of length larger than $1$.

Given diagrams $D_1$ and $D_2$, we define their \textit{direct sum} to be the skew diagram $D= D_1 \oplus D_2$ that consists of the subdiagrams $D_1$ and $D_2$ such that the top-right box of $D_1$ is one step left and one step down from the bottom-left box of $D_{2}$. 
Further, given any diagram $D$, the $180^{\circ}$ rotation of a diagram $D$ is denoted by $D^{\circ}$.

\begin{example} Let $D_1 = (2,2,2)/(1,1)$ and $D_2 = (4,4,2)$. Then ${D_1}^\circ$ is the hook given by $(2,1,1)$. We display the direct sum $D_1 \oplus D_2$. 

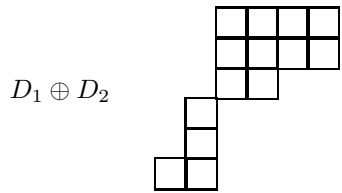
\setlength{\unitlength}{0.4mm}

\begin{picture}(120,50)(-35,-5)

\put(32,10){$D_1 \oplus D_2$}

\put(100,30){\framebox(10,10)[tl]{ }}
\put(110,30){\framebox(10,10)[tl]{ }}
\put(120,30){\framebox(10,10)[tl]{ }}
\put(130,30){\framebox(10,10)[tl]{ }}

\put(100,20){\framebox(10,10)[tl]{ }}
\put(110,20){\framebox(10,10)[tl]{ }}
\put(120,20){\framebox(10,10)[tl]{ }}
\put(130,20){\framebox(10,10)[tl]{ }}

\put(100,10){\framebox(10,10)[tl]{ }}
\put(110,10){\framebox(10,10)[tl]{ }}

\put(90,0){\framebox(10,10)[tl]{ }}
\put(90,-10){\framebox(10,10)[tl]{ }}
\put(90,-20){\framebox(10,10)[tl]{ }}
\put(80,-20){\framebox(10,10)[tl]{ }}

\end{picture}

\end{example}

\

If $D$ is a diagram, then a \textit{tableau}---plural \textit{tableaux}---$\mathcal{T}$ \textit{of shape $D$} is obtained by filling the boxes of the $D$ with the positive integers. 
It is a \textit{semistandard Young tableau} (SSYT---plural SSYTx) if each row of $\mathcal{T}$ gives a weakly increasing sequence of integers and each column of $\mathcal{T}$ gives a strictly increasing sequence of integers.
The \textit{content} of a tableau $\mathcal{T}$ is the weak composition given by
\[ \nu(\mathcal{T}) = ( \# \textrm{1's in }\mathcal{T}, \# \textrm{2's in }\mathcal{T}, \ldots).\]

Given a skew diagram $D$, the \textit{skew Schur function corresponding to $D$} is defined to be \label{slambda}
\begin{equation}
\label{schurdef}
 s_{D}(\textbf{x}) = \sum_{\mathcal{T}} {x_1}^{\# \textrm{1's in } \mathcal{T}}{x_2}^{\# \textrm{2's in } \mathcal{T}} \cdots,
\end{equation}
where the sum is taken over all semistandard Young tableaux $\mathcal{T}$ of shape $D$. 
When $D=\lambda$ is a partition, $s_\lambda$ is called the \textit{Schur function corresponding to $\lambda$}.

The set $\{s_\lambda | \lambda \vdash n \}$ is a basis of $\Lambda^n$, the set of homogeneous symmetric functions of degree $n$. 
Therefore for each $f \in \Lambda^n$ we can write $f=\sum_{\lambda} a_\lambda s_\lambda$ for appropriate coefficients. 
For any partitions $\mu$ and $\nu$ we have
\begin{equation}
\label{smusnu}
 s_\mu s_\nu = \sum_{\lambda \vdash n} c_{\mu \nu}^{\lambda} s_\lambda, 
\end{equation}
and for any skew diagram $\lambda / \mu$ we have
\begin{equation}
\label{slambdaskewmu}
 s_{\lambda / \mu} = \sum_{\nu \vdash n} c_{\mu \nu}^{\lambda} s_\nu 
\end{equation}
where the $c_{\mu \nu}^{\lambda}$ are the \textit{Littlewood-Richardson coefficients}.
The Littlewood-Richardson coefficients are non-negative integers and count an interesting class of SSYT that we now describe.

Given a tableau $\mathcal{T}$, the \textit{reading word} of $\mathcal{T}$ is the sequence of integers obtained by reading the entries of the rows of $\mathcal{T}$ from right to left, proceeding from the top row to the bottom. 
We say that a sequence $r=r_1,r_2,\ldots, r_k$ is \textit{lattice} if, for each $j$, when reading the sequence from left to right the number of $j$'s that we have read is never less than the number of $j+1$'s that we have read.

\begin{theorem} [Littlewood-Richardson Rule] (\cite{LR})
\label{lr}

For partitions $\lambda, \mu$, and $\nu$, the \textit{Littlewood-Richardson coefficient} $c_{\mu \nu}^{\lambda}$ is the number of SSYTx of shape $\lambda / \mu$, content $\nu$, with lattice reading word.
\end{theorem}

For any $f = \sum_{\lambda \vdash n} a_{\lambda} s_{\lambda} \in \Lambda^n$, we say that $f$ is \textit{Schur-positive}, and write $f \geq_s 0$, if each $a_{\lambda} \geq 0$. 
The Littlewood-Richardson rule shows that both $s_{\mu} s_{\nu}$ and $s_{\lambda / \mu}$ are Schur-positive. 
For $f,g \in \Lambda^n$, we will be interested in whether or not the difference $f-g$ is Schur positive. 
We shall write $f \geq_s g$ whenever $f-g$ is Schur-positive. 
If neither $f-g$ nor $g-f$ is Schur-positive we say that $f$ and $g$ are \textit{Schur-incomparable}. 
Further, we write $D_1 \succeq_s D_2$ if $s_{D_1} \geq_s s_{D_2}$.

If we consider the relation $\succeq_s$ on the set of all Schur-equivalent classes of diagrams (i.e. $[D]_s= \{D' | s_{D}=s_{D'}\}$), then $\succeq_s$ defines a partial ordering.
This allows us to view the Hasse diagram for the relation $\succeq_s$ on the set of these Schur-equivalent classes.
Some work in determining these equivalence classes includes \cite{{btvw},{g3},{mvw1},{rsvw}}.

\

We close these preliminaries by mentioning two useful results regarding skew Schur functions.
\begin{theorem}{(\cite{stanley}, Exercise 7.56(a))}
\label{rotate}
Given a skew diagram $D$, 
\begin{equation} s_D= s_{D^{\circ}}. 
\end{equation}
\end{theorem}

\begin{theorem}
\label{disjprod}
The Schur function of any disconnected skew diagram is reducible. If $D = D_1 \oplus D_2$, then we have 
\begin{equation}
s_{D}=s_{D_1} s_{D_2}.
\end{equation}
\end{theorem}
\begin{proof} Any SSYT of shape $D_1 \oplus D_2$ gives rise to SSYTx of shape $D_1$ and $D_2$ by restricting to the subdiagrams $D_1$ and $D_2$. 
Conversely, any pair of SSYTx $\mathcal{T}_1$ of shape $D_1$ and $\mathcal{T}_2$ of shape $D_2$ give rise to the tableau $\mathcal{T}_1 \oplus \mathcal{T}_2$ of shape $D_1 \oplus D_2$, which is clearly semistandard. \qed
\end{proof}

A thorough study of this material can be found in sources such as \cite{sagan} or \cite{stanley}.

\section{Staircases and Fat Staircases}

A Ferrers diagram is a \textit{staircase} if it is the Ferrers diagram of a partition of the form $\lambda = (n,n-1,n-2,\ldots, 2,1)$ or if it is the $180^{\circ}$ rotation of such a diagram. 
Both these diagrams are referred to as \textit{staircases of length $n$} and will be denoted by $\delta_n$ and $\Delta_n$ respectively.

\begin{example} Here we see the two staircases of length 5.

\

\

\

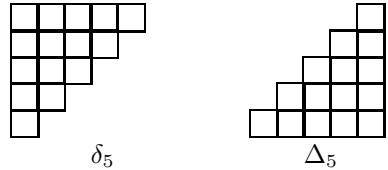
\setlength{\unitlength}{0.35mm}

\begin{picture}(100,50)(-80,-5)

\put(-10,50){\framebox(10,10)[tl]{ }}
\put(0,50){\framebox(10,10)[tl]{ }}
\put(10,50){\framebox(10,10)[tl]{ }}
\put(20,50){\framebox(10,10)[tl]{ }}
\put(30,50){\framebox(10,10)[tl]{ }}
     
\put(-10,40){\framebox(10,10)[tl]{ }}
\put(0,40){\framebox(10,10)[tl]{ }}
\put(10,40){\framebox(10,10)[tl]{ }}
\put(20,40){\framebox(10,10)[tl]{ }}
 
\put(-10,30){\framebox(10,10)[tl]{ }}
\put(0,30){\framebox(10,10)[tl]{ }}
\put(10,30){\framebox(10,10)[tl]{ }}

\put(-10,20){\framebox(10,10)[tl]{ }}
\put(0,20){\framebox(10,10)[tl]{ }}

\put(-10,10){\framebox(10,10)[tl]{ }}

\put(20,0){$\delta_5$}

\put(100,0){$\Delta_5$}

\put(120,50){\framebox(10,10)[tl]{ }}
\put(120,40){\framebox(10,10)[tl]{ }}
\put(120,30){\framebox(10,10)[tl]{ }}
\put(120,20){\framebox(10,10)[tl]{ }}
\put(120,10){\framebox(10,10)[tl]{ }}

\put(110,40){\framebox(10,10)[tl]{ }}
\put(110,30){\framebox(10,10)[tl]{ }}
\put(110,20){\framebox(10,10)[tl]{ }}
\put(110,10){\framebox(10,10)[tl]{ }}

\put(100,30){\framebox(10,10)[tl]{ }}
\put(100,20){\framebox(10,10)[tl]{ }}
\put(100,10){\framebox(10,10)[tl]{ }}

\put(90,20){\framebox(10,10)[tl]{ }}
\put(90,10){\framebox(10,10)[tl]{ }}

\put(80,10){\framebox(10,10)[tl]{ }}

\end{picture}

\end{example}

Given a composition $\alpha=(\alpha_1, \ldots,\alpha_n)$, we let 
\[ \delta_\alpha= (n^{\alpha_n},{n-1}^{\alpha_{n-1}}, \ldots,2^{\alpha_2}, 1^{\alpha_1}) \textrm{ and } \Delta_\alpha= (n^{\alpha_n},{n-1}^{\alpha_{n-1}}, \ldots,2^{\alpha_2}, 1^{\alpha_1})^\circ. \]
We call a skew diagram $D$ a \textit{fat staircase} if $D=\delta_{\alpha}$ or $D=\Delta_{\alpha}$ for some composition $\alpha$.
The numbers $\alpha_i$ count the number of rows of $D$ with $i$ boxes, for each $i$.
Using this notation the regular staircases may be expressed as $\delta_n = \delta_{(1^n)}$ and $\Delta_n = \Delta_{(1^n)}$, respectively. 
Both fat staircases $\delta_\alpha$ and $\Delta_\alpha$ have width $= l(\alpha)$ and length $= |\alpha| = \sum_{i=1}^{n} \alpha_i$. 

\begin{example}
Here we see the the fat staircases $\delta_{(1,2,2)}$ and $\Delta_{(3,1,2,3)}$.

\

\

\

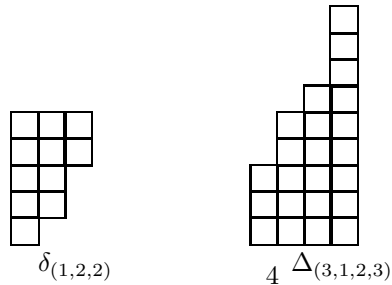
\setlength{\unitlength}{0.35mm}

\begin{picture}(100,50)(-50,-15)

\put(10,10){\framebox(10,10)[tl]{ }}
\put(20,10){\framebox(10,10)[tl]{ }}
\put(30,10){\framebox(10,10)[tl]{ }}
     
\put(10,0){\framebox(10,10)[tl]{ }}
\put(20,0){\framebox(10,10)[tl]{ }}
\put(30,0){\framebox(10,10)[tl]{ }}
 
\put(10,-10){\framebox(10,10)[tl]{ }}
\put(20,-10){\framebox(10,10)[tl]{ }}

\put(10,-20){\framebox(10,10)[tl]{ }}
\put(20,-20){\framebox(10,10)[tl]{ }}

\put(10,-30){\framebox(10,10)[tl]{ }}

\put(20,-40){$\delta_{(1,2,2)}$}

\put(115,-40){$\Delta_{(3,1,2,3)}$}

\put(100,-30){\framebox(10,10)[tl]{ }}
\put(110,-30){\framebox(10,10)[tl]{ }}
\put(120,-30){\framebox(10,10)[tl]{ }}
\put(130,-30){\framebox(10,10)[tl]{ }}

\put(100,-20){\framebox(10,10)[tl]{ }}
\put(110,-20){\framebox(10,10)[tl]{ }}
\put(120,-20){\framebox(10,10)[tl]{ }}
\put(130,-20){\framebox(10,10)[tl]{ }}

\put(100,-10){\framebox(10,10)[tl]{ }}
\put(110,-10){\framebox(10,10)[tl]{ }}
\put(120,-10){\framebox(10,10)[tl]{ }}
\put(130,-10){\framebox(10,10)[tl]{ }}

\put(110,0){\framebox(10,10)[tl]{ }}
\put(120,0){\framebox(10,10)[tl]{ }}
\put(130,0){\framebox(10,10)[tl]{ }}

\put(110,10){\framebox(10,10)[tl]{ }}
\put(120,10){\framebox(10,10)[tl]{ }}
\put(130,10){\framebox(10,10)[tl]{ }}

\put(120,20){\framebox(10,10)[tl]{ }}
\put(130,20){\framebox(10,10)[tl]{ }}

\put(130,30){\framebox(10,10)[tl]{ }}

\put(130,40){\framebox(10,10)[tl]{ }}

\put(130,50){\framebox(10,10)[tl]{ }}

\end{picture}

\end{example}


\

\

\

Given a composition $\alpha$, $k \geq 0$, and a partition $\lambda$ with $\lambda_1 -k \leq l(\alpha)$ we now define $\mathcal{S}(\lambda, \alpha;k)$ to be the diagram obtained by placing $\lambda$ immediately below $\Delta_{\alpha}$ such that the rows of the two diagrams overlap in precisely $\lambda_1 -k$ positions.
We call $\mathcal{S}(\lambda, \alpha;k)$ a \textit{fat staircase with bad foundation}.
The subdiagram $\lambda$ is called the foundation of $\mathcal{S}(\lambda,\alpha;k)$.

The fact that $\Delta_{\alpha}$ and $\lambda$ overlap in precisely $\lambda_1 -k$ positions means that the first row of $\lambda$ begins exactly one box below and $k$ boxes left of the bottom-left box of the diagram $\Delta_{\alpha}$.

\begin{example} If we take $\alpha = (1,1,3,1,2,1)$, $\lambda = (6,5,5,5,3)$, and $k=0$, then we obtain the following staircase with bad foundation $\mathcal{S}(\lambda, \alpha;k)$.

\

\

\

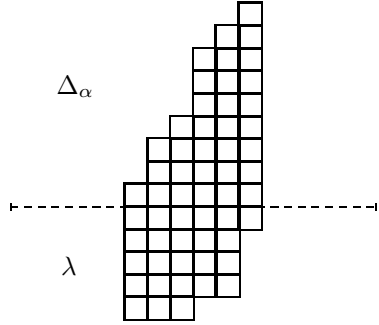
\setlength{\unitlength}{0.3mm}

\begin{picture}(000,140)(110,-80)

\put(180,30){$\Delta_{\alpha}$}

\put(182,-50){$\lambda$}

\put(160,-20){\dashbox{3}(160,0)[tl]{ }}

\put(210,-70){\framebox(10,10)[tl]{ }}
\put(220,-70){\framebox(10,10)[tl]{ }}
\put(230,-70){\framebox(10,10)[tl]{ }}

\put(210,-60){\framebox(10,10)[tl]{ }}
\put(220,-60){\framebox(10,10)[tl]{ }}
\put(230,-60){\framebox(10,10)[tl]{ }}
\put(240,-60){\framebox(10,10)[tl]{ }}
\put(250,-60){\framebox(10,10)[tl]{ }}

\put(210,-50){\framebox(10,10)[tl]{ }}
\put(220,-50){\framebox(10,10)[tl]{ }}
\put(230,-50){\framebox(10,10)[tl]{ }}
\put(240,-50){\framebox(10,10)[tl]{ }}
\put(250,-50){\framebox(10,10)[tl]{ }}

\put(210,-40){\framebox(10,10)[tl]{ }}
\put(220,-40){\framebox(10,10)[tl]{ }}
\put(230,-40){\framebox(10,10)[tl]{ }}
\put(240,-40){\framebox(10,10)[tl]{ }}
\put(250,-40){\framebox(10,10)[tl]{ }}

\put(210,-30){\framebox(10,10)[tl]{ }}
\put(220,-30){\framebox(10,10)[tl]{ }}
\put(230,-30){\framebox(10,10)[tl]{ }}
\put(240,-30){\framebox(10,10)[tl]{ }}
\put(250,-30){\framebox(10,10)[tl]{ }}
\put(260,-30){\framebox(10,10)[tl]{ }}

\put(210,-20){\framebox(10,10)[tl]{ }}
\put(220,-20){\framebox(10,10)[tl]{ }}
\put(230,-20){\framebox(10,10)[tl]{ }}
\put(240,-20){\framebox(10,10)[tl]{ }}
\put(250,-20){\framebox(10,10)[tl]{ }}
\put(260,-20){\framebox(10,10)[tl]{ }}

\put(220,-10){\framebox(10,10)[tl]{ }}
\put(230,-10){\framebox(10,10)[tl]{ }}
\put(240,-10){\framebox(10,10)[tl]{ }}
\put(250,-10){\framebox(10,10)[tl]{ }}
\put(260,-10){\framebox(10,10)[tl]{ }}

\put(220,0){\framebox(10,10)[tl]{ }}
\put(230,0){\framebox(10,10)[tl]{ }}
\put(240,0){\framebox(10,10)[tl]{ }}
\put(250,0){\framebox(10,10)[tl]{ }}
\put(260,0){\framebox(10,10)[tl]{ }}

\put(230,10){\framebox(10,10)[tl]{ }}
\put(240,10){\framebox(10,10)[tl]{ }}
\put(250,10){\framebox(10,10)[tl]{ }}
\put(260,10){\framebox(10,10)[tl]{ }}

\put(240,20){\framebox(10,10)[tl]{ }}
\put(250,20){\framebox(10,10)[tl]{ }}
\put(260,20){\framebox(10,10)[tl]{ }}

\put(240,30){\framebox(10,10)[tl]{ }}
\put(250,30){\framebox(10,10)[tl]{ }}
\put(260,30){\framebox(10,10)[tl]{ }}

\put(240,40){\framebox(10,10)[tl]{ }}
\put(250,40){\framebox(10,10)[tl]{ }}
\put(260,40){\framebox(10,10)[tl]{ }}

\put(250,50){\framebox(10,10)[tl]{ }}
\put(260,50){\framebox(10,10)[tl]{ }}

\put(260,60){\framebox(10,10)[tl]{ }}

\end{picture}

\end{example}

One of the advantages in computing the skew Schur functions of fat staircases with bad foundations is that, when using the Littlewood-Richarson rule, the $\Delta_\alpha$ portion of the diagram can be filled in only one way. 
By using Theorem~\ref{rotate} we can be see this algebraically from the equation 
\[ s_{\Delta_\alpha} = s_{\Delta_\alpha^\circ} = s_{\delta_\alpha}, \]
 where $s_{\delta_\alpha}$ is a Schur function. 
The unique filling of $\Delta_\alpha$ obeying the semistandard conditions and the lattice condition is easily seen to be the filling that places the entries $1,2,\ldots, l$ into each column of length $l$.


\setlength{\unitlength}{0.5mm}
\begin{picture}(0,0)(0,0)
\put(90,-74){\line(0,1){26} }
\end{picture}
\begin{lemma} 

\label{kfatfirstrowlemma}
Let $\mathcal{S}(\lambda,\alpha;k)$ be a fat staircase with bad foundation for some $k \geq 0$ and $\mathcal{T}$ be a SSYT of shape $\mathcal{S}(\lambda,\alpha;k)$ whose reading word is lattice.
If $\alpha=(\alpha_1,\dots, \alpha_n)$, then the entries in the first row of the foundation of $\mathcal{T}$ consist values taken from the set 
\[R_{\alpha,k} = \left\{ 1+\sum_{i=1}^{j} \alpha_{n+1-i} \textrm{ } \textrm{ } \textrm{ } j= 1, 2, \ldots, n  \right\} \cup \left\{ \begin{array}{cll}
\{1\} & \mbox{if} & k > 0 \\
\emptyset &  \mbox{if} & k=0. 
\end{array}\right.\]
Furthermore, the value $1$ can occur at most $k$ times and the rest of the values can appear at most once.
\end{lemma}

\begin{proof}

Let $R$ be the first row of the foundation of $\mathcal{T}$ and $t \in R$.

Since $\mathcal{T}$ is a SSYT, the columns strictly increase. Thus $t = 1$ is allowed if and only if $k \geq 1$ since it is precisely in that case that the first value in $R$ is not below an entry of $\Delta_{\alpha}$. 
Furthermore, since there are only $k$ boxes from the first row of the foundation of $\mathcal{T}$ that extend out from $\Delta_\alpha$, there can be at most $k$ $1$'s in $R$. 

If $t > 1$ then, when reading the row $R$ from right to left, the lattice condition implies that there is at least one more $t-1$ in $\Delta_{\alpha}$ than there are $t$'s in $\Delta_{\alpha}$. 
Since the content of $\Delta_{\alpha}$ is $(n^{\alpha_n},{n-1}^{\alpha_{n-1}}, \ldots, 1^{\alpha_1})$, the only instances when this occurs are when $t=1+\sum_{i=1 \ldots j} \alpha_{n+1-j}$ for $j=1,2,\ldots,n$.
Therefore every entry of $R$ is an element of $R_{\alpha,k}$.
Further, if a value $t>1$ appeared twice in $R$, then the lattice condition would be violated. 
Hence each $t \in R_{\alpha,k}$, $t \neq 1$, can appear at most once in $R$. \qed

\end{proof}

The next result tells us when we may obtain a SSYT of shape $\mathcal{S}(\lambda,\alpha;k)$ with lattice reading word from a SSYT of shape $\lambda  \oplus \Delta_\alpha$ with lattice reading word. 

\begin{lemma} 
\label{kfatjoinlemma}
Let $\alpha$ be a composition, $\lambda$ be a partition, and $k \geq 0$ such that $\lambda_1 -k \leq l(\alpha)$.
If $T$ is a SSYT of shape $\lambda  \oplus \Delta_\alpha$ with lattice reading word such that there are at most $k$ $1$'s in the first row of $\lambda$, then the tableau of shape $\mathcal{S}(\lambda,\alpha;k)$ obtained from $T$ by shifting the foundation $\lambda$ to the right is also a SSYT with lattice reading word.

\end{lemma}

\begin{proof}
Let $T$ be a SSYT of shape $\lambda  \oplus \Delta_\alpha$ with lattice reading word and let $T_k'$ be the tableau of shape $\mathcal{S}(\lambda,\alpha;k)$ obtained from $T$ by shifting the foundation $\lambda$ to the right.
Since shifting $\lambda$ to the right does not affect the order in which the entries are read, $T_k'$ has a lattice reading word.
Also, the rows of $T_k'$ weakly increase since they are the same as the rows of $T$.
Further, to check that the columns of $T_k'$ strictly increase, we need only check that they strictly increase at the positions where the two subdiagrams $\Delta_{\alpha}$ and $\lambda$ are joined.

Let $R$ denote the first row of $\lambda$ and $\alpha= (\alpha_1, \ldots, \alpha_n)$. 
As in the proof of Lemma~\ref{kfatfirstrowlemma}, the lattice condition on $T$ implies that the entries of $R$ consist of values of $R_{\alpha, k}$. 
Further, the value $1$ can occur at most $k$ times and the rest of the values of $R$ are distinct.
Let $q$ be the number of times $1$ appears in $R$, so that $k \geq q$. 
Further, let $r_1 \leq r_2 \leq \ldots \leq r_{\lambda_1}$ be the entries of $R$. 

\

Consider the case $k \geq 1$. 
Since $r_1=r_2 =\ldots = r_q =1$, we have $r_q = \textrm{min}(R_{\alpha,k})$ and for each $1 \leq j \leq n$ we have  
\[ r_{j+q} \geq \textrm{ the } (j+1) \textrm{-th smallest value of } R_{\alpha,k} = 1 + \sum_{i=1}^{j} \alpha_{n+1-i}. \] 
Since $k \geq q$, for each $1 \leq j \leq n$ we have 
\[ r_{j+k} \geq r_{j+q} \geq 1 + \sum_{i=1}^{j} \alpha_{n+1-i}. \]
As illustrated in the diagram below, the entry $r_{j+k}$ is beneath $\sum_{i=1}^{j} \alpha_{n+1-i}$ boxes. 
From the unique filling of $\Delta_{\alpha}$, the entry of $\Delta_\alpha$ directly above $r_{j+k}$ is $\sum_{i=1}^{j} \alpha_{n+1-i}$.

\

\

\

\

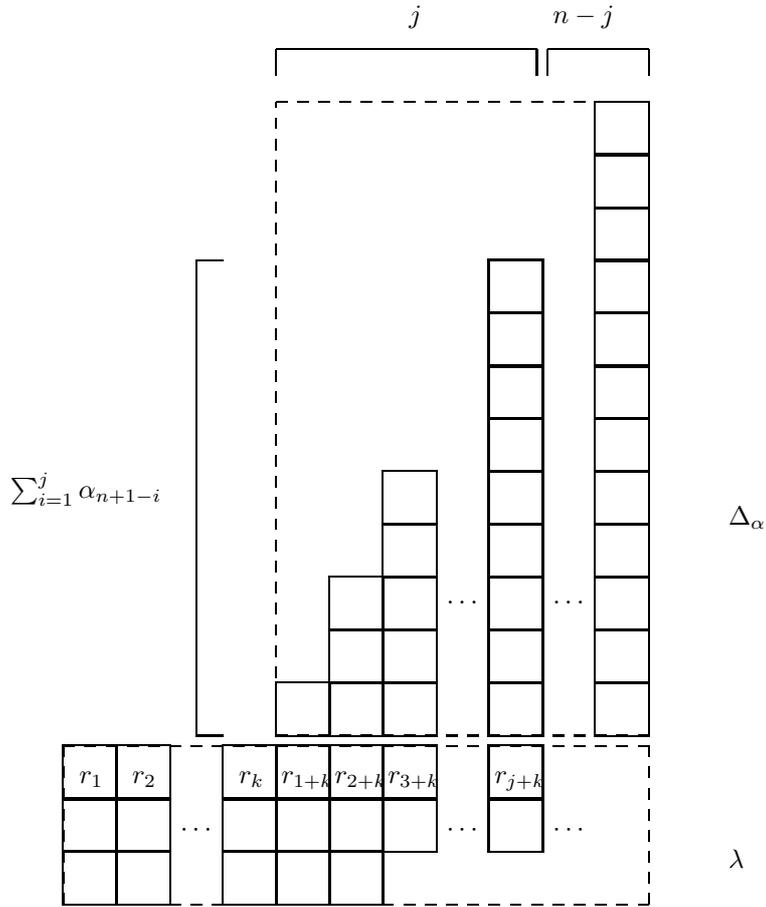
\setlength{\unitlength}{0.7mm}

\begin{picture}(100,80)(30,5)

\put(115,105){$j$}

\put(142,105){$n-j$}

\put(40,15){$\sum_{i=1}^{j} \alpha_{n+1-i}$}

\put(75,-30){\line(0,1){90}}

\put(75,-30){\line(1,0){5}}
\put(75,60){\line(1,0){5}}

\put(90,100){\line(1,0){49}}
\put(141,100){\line(1,0){19}}

\put(90,100){\line(0,-1){5}}
\put(139,100){\line(0,-1){5}}

\put(141,100){\line(0,-1){5}}
\put(160,100){\line(0,-1){5}}

\put(90,-30){\dashbox{2}(70,120)[tl]{ }}

\put(50,-62){\dashbox{2}(110,30)[tl]{ }}

\put(175,10){$\Delta_\alpha$}
\put(175,-55){$\lambda $}

\put(90,-30){\framebox(10,10)[tl]{ }}
\put(100,-30){\framebox(10,10)[tl]{ }}
\put(110,-30){\framebox(10,10)[tl]{ }}
\put(130,-30){\framebox(10,10)[tl]{ }}
\put(150,-30){\framebox(10,10)[tl]{ }}

\put(100,-20){\framebox(10,10)[tl]{ }}
\put(110,-20){\framebox(10,10)[tl]{ }}
\put(130,-20){\framebox(10,10)[tl]{ }}
\put(150,-20){\framebox(10,10)[tl]{ }}

\put(100,-10){\framebox(10,10)[tl]{ }}
\put(110,-10){\framebox(10,10)[tl]{ }}
\put(130,-10){\framebox(10,10)[tl]{ }}
\put(150,-10){\framebox(10,10)[tl]{ }}

\put(110,0){\framebox(10,10)[tl]{ }}
\put(130,0){\framebox(10,10)[tl]{ }}
\put(150,0){\framebox(10,10)[tl]{ }}

\put(110,10){\framebox(10,10)[tl]{ }}
\put(130,10){\framebox(10,10)[tl]{ }}
\put(150,10){\framebox(10,10)[tl]{ }}

\put(130,20){\framebox(10,10)[tl]{ }}
\put(150,20){\framebox(10,10)[tl]{ }}

\put(130,30){\framebox(10,10)[tl]{ }}
\put(150,30){\framebox(10,10)[tl]{ }}

\put(130,40){\framebox(10,10)[tl]{ }}
\put(150,40){\framebox(10,10)[tl]{ }}

\put(130,50){\framebox(10,10)[tl]{ }}
\put(150,50){\framebox(10,10)[tl]{ }}

\put(150,60){\framebox(10,10)[tl]{ }}

\put(150,70){\framebox(10,10)[tl]{ }}

\put(150,80){\framebox(10,10)[tl]{ }}

\put(91,-39){$r_{1+k}$}
\put(101,-39){$r_{2+k}$}
\put(111,-39){$r_{3+k}$}
\put(131,-39){$r_{j+k}$}

\put(53,-39){$r_1$}
\put(63,-39){$r_2$}
\put(72,-49){$\cdots$}
\put(83,-39){$r_k$}

\put(50,-42){\framebox(10,10)[tl]{ }}
\put(60,-42){\framebox(10,10)[tl]{ }}
\put(80,-42){\framebox(10,10)[tl]{ }}

\put(90,-42){\framebox(10,10)[tl]{ }}
\put(100,-42){\framebox(10,10)[tl]{ }}
\put(110,-42){\framebox(10,10)[tl]{ }}
\put(130,-42){\framebox(10,10)[tl]{ }}
     
\put(50,-52){\framebox(10,10)[tl]{ }}
\put(60,-52){\framebox(10,10)[tl]{ }}
\put(80,-52){\framebox(10,10)[tl]{ }}

\put(90,-52){\framebox(10,10)[tl]{ }}
\put(100,-52){\framebox(10,10)[tl]{ }}
\put(110,-52){\framebox(10,10)[tl]{ }}
\put(130,-52){\framebox(10,10)[tl]{ }}

\put(50,-62){\framebox(10,10)[tl]{ }}
\put(60,-62){\framebox(10,10)[tl]{ }}
\put(80,-62){\framebox(10,10)[tl]{ }}

\put(90,-62){\framebox(10,10)[tl]{ }}
\put(100,-62){\framebox(10,10)[tl]{ }}

\put(122,-49){$\cdots$}
\put(142,-49){$\cdots$}

\put(122,-5){\ldots}
\put(142,-5){\ldots}

\end{picture}

\

\

\

\
\

\

\

\

\

\

\

\

\

Thus the columns strictly increase. 
Therefore $T_1'$ is a SSYT with lattice reading word, as desired.

\

Now consider the case when $k=0$. Then for each $1 \leq j \leq n$ we have 
\[ r_j \geq j \textrm{-th smallest value of } R_{\alpha, k}  \geq 1 + \sum_{i=1}^{j} \alpha_{n+1-i}. \]
Also, the entry $r_j$ is beneath precisely $\sum_{i=1}^{j} \alpha_{n+1-i}$ boxes, so the entry of $\Delta_\alpha$ directly above $r_j$ is $\sum_{i=1}^{j} \alpha_{n+1-i}$. 
Thus the columns strictly increase. 
Therefore $T_k'$ is a SSYT with lattice reading word, as desired.
\qed

\end{proof}

\begin{example}
Let $\alpha = (2,2,1)$, $\lambda = (3,2)$, and $k=2$.
Consider the SSYT of shape $\lambda \oplus \Delta_\alpha$ with lattice reading word and two $1$'s in the first row of $\lambda$ shown on the left.
This gives rise to the SSYT of shape $\mathcal{S}(\lambda,\alpha;2)$ with lattice reading word shown on the right.

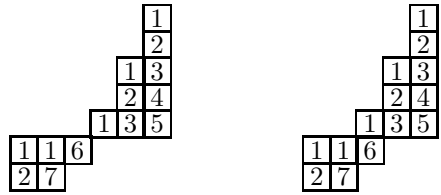
\setlength{\unitlength}{0.35mm}

\begin{picture}(0,100)(-30,10)

\put(70,90){\framebox(10,10)[c]{\textrm{ }$1$ }}
\put(70,80){\framebox(10,10)[c]{\textrm{ }$2$ } }
\put(70,70){\framebox(10,10)[c]{\textrm{ }$3$ } }
\put(70,60){\framebox(10,10)[c]{\textrm{ }$4$ } }
\put(70,50){\framebox(10,10)[c]{\textrm{ }$5$ } }

\put(60,70){\framebox(10,10)[c]{\textrm{ }$1$ } }
\put(60,60){\framebox(10,10)[c]{\textrm{ }$2$ } }
\put(60,50){\framebox(10,10)[c]{\textrm{ }$3$ } }

\put(50,50){\framebox(10,10)[c]{\textrm{ }$1$ } }

\put(20,40){\framebox(10,10)[c]{\textrm{ }$1$ } }
\put(30,40){\framebox(10,10)[c]{\textrm{ }$1$ } }
\put(40,40){\framebox(10,10)[c]{\textrm{ }$6$ } }

\put(20,30){\framebox(10,10)[c]{\textrm{ }$2$ } }
\put(30,30){\framebox(10,10)[c]{\textrm{ }$7$ } }

\put(170,90){\framebox(10,10)[c]{\textrm{ }$1$ }}
\put(170,80){\framebox(10,10)[c]{\textrm{ }$2$ } }
\put(170,70){\framebox(10,10)[c]{\textrm{ }$3$ } }
\put(170,60){\framebox(10,10)[c]{\textrm{ }$4$ } }
\put(170,50){\framebox(10,10)[c]{\textrm{ }$5$ } }

\put(160,70){\framebox(10,10)[c]{\textrm{ }$1$ } }
\put(160,60){\framebox(10,10)[c]{\textrm{ }$2$ } }
\put(160,50){\framebox(10,10)[c]{\textrm{ }$3$ } }

\put(150,50){\framebox(10,10)[c]{\textrm{ }$1$ } }

\put(130,40){\framebox(10,10)[c]{\textrm{ }$1$ } }
\put(140,40){\framebox(10,10)[c]{\textrm{ }$1$ } }
\put(150,40){\framebox(10,10)[c]{\textrm{ }$6$ } }

\put(130,30){\framebox(10,10)[c]{\textrm{ }$2$ } }
\put(140,30){\framebox(10,10)[c]{\textrm{ }$7$ } }

\end{picture}

\end{example}

\section{Fat Staircases with Hook Foundations}

Recall from the introduction, that we write $D_1 \succeq_s D_2$ whenever $s_{D_1}-s_{D_2} \geq_s 0$. 
If we consider the relation $\succeq_s$ on the set of all Schur-equivalent classes of diagrams (i.e. $[D]_s= \{D' | s_{D}=s_{D'}\}$), then $\succeq_s$ defines a partial ordering.
This allows us to view the Hasse diagram for the relation $\succeq_s$ on the set of these Schur-equivalent classes.
For the sake of convenience, we write $D$ in place of $[D]_s$.

\begin{example}

Here we show the Hasse diagram for $\succeq_s$ on the collection of staircases with bad foundations $\mathcal{S}(\lambda,(1^7);0)$, for $\lambda$ varying over all hooks of size $7$.
A line drawn from a diagram $D_1$ to a diagram $D_2$ in an upwards direction indicates that $s_{D_1} - s_{D_2} \geq_s 0$.
We note that the diagrams along the top are all Schur-incomparable. That is, they form an anti-chain with regards to $\succeq_s$.
Also, the diagrams along the right are all comparable. That is, they form a chain with regards to $\succeq_s$.

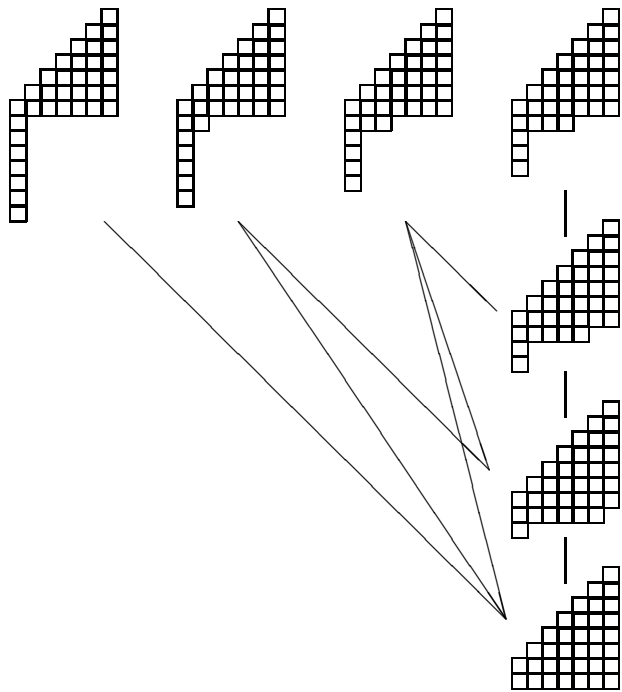
\setlength{\unitlength}{0.20mm}

\begin{picture}(00,500)(-100,-400)

\put(42,-80){\line(1,-1){263}}
\put(130,-80){\line(1,-1){165}}
\put(130,-80){\line(2,-3){176}}
\put(240,-80){\line(1,-1){60}}
\put(240,-80){\line(1,-3){55}}
\put(240,-80){\line(1,-4){66}}

\put(345,-320){\line(0,1){30}}
\put(345,-210){\line(0,1){30}}
\put(345,-90){\line(0,1){30}}

\put(310,-390){\framebox(10,10)[tl]{ }}
\put(320,-390){\framebox(10,10)[tl]{ }}
\put(330,-390){\framebox(10,10)[tl]{ }}
\put(340,-390){\framebox(10,10)[tl]{ }}
\put(350,-390){\framebox(10,10)[tl]{ }}
\put(360,-390){\framebox(10,10)[tl]{ }}
\put(370,-390){\framebox(10,10)[tl]{ }}

\put(310,-380){\framebox(10,10)[tl]{ }}
\put(320,-380){\framebox(10,10)[tl]{ }}
\put(330,-380){\framebox(10,10)[tl]{ }}
\put(340,-380){\framebox(10,10)[tl]{ }}
\put(350,-380){\framebox(10,10)[tl]{ }}
\put(360,-380){\framebox(10,10)[tl]{ }}
\put(370,-380){\framebox(10,10)[tl]{ }}

\put(320,-370){\framebox(10,10)[tl]{ }}
\put(330,-370){\framebox(10,10)[tl]{ }}
\put(340,-370){\framebox(10,10)[tl]{ }}
\put(350,-370){\framebox(10,10)[tl]{ }}
\put(360,-370){\framebox(10,10)[tl]{ }}
\put(370,-370){\framebox(10,10)[tl]{ }}

\put(330,-360){\framebox(10,10)[tl]{ }}
\put(340,-360){\framebox(10,10)[tl]{ }}
\put(350,-360){\framebox(10,10)[tl]{ }}
\put(360,-360){\framebox(10,10)[tl]{ }}
\put(370,-360){\framebox(10,10)[tl]{ }}
     
\put(340,-350){\framebox(10,10)[tl]{ }}
\put(350,-350){\framebox(10,10)[tl]{ }}
\put(360,-350){\framebox(10,10)[tl]{ }}
\put(370,-350){\framebox(10,10)[tl]{ }}
 
\put(350,-340){\framebox(10,10)[tl]{ }}
\put(360,-340){\framebox(10,10)[tl]{ }}
\put(370,-340){\framebox(10,10)[tl]{ }}

\put(360,-330){\framebox(10,10)[tl]{ }}
\put(370,-330){\framebox(10,10)[tl]{ }}

\put(370,-320){\framebox(10,10)[tl]{ }}

\put(310,-290){\framebox(10,10)[tl]{ }}
\put(310,-280){\framebox(10,10)[tl]{ }}
\put(320,-280){\framebox(10,10)[tl]{ }}
\put(330,-280){\framebox(10,10)[tl]{ }}
\put(340,-280){\framebox(10,10)[tl]{ }}
\put(350,-280){\framebox(10,10)[tl]{ }}
\put(360,-280){\framebox(10,10)[tl]{ }}

\put(310,-270){\framebox(10,10)[tl]{ }}
\put(320,-270){\framebox(10,10)[tl]{ }}
\put(330,-270){\framebox(10,10)[tl]{ }}
\put(340,-270){\framebox(10,10)[tl]{ }}
\put(350,-270){\framebox(10,10)[tl]{ }}
\put(360,-270){\framebox(10,10)[tl]{ }}
\put(370,-270){\framebox(10,10)[tl]{ }}

\put(320,-260){\framebox(10,10)[tl]{ }}
\put(330,-260){\framebox(10,10)[tl]{ }}
\put(340,-260){\framebox(10,10)[tl]{ }}
\put(350,-260){\framebox(10,10)[tl]{ }}
\put(360,-260){\framebox(10,10)[tl]{ }}
\put(370,-260){\framebox(10,10)[tl]{ }}

\put(330,-250){\framebox(10,10)[tl]{ }}
\put(340,-250){\framebox(10,10)[tl]{ }}
\put(350,-250){\framebox(10,10)[tl]{ }}
\put(360,-250){\framebox(10,10)[tl]{ }}
\put(370,-250){\framebox(10,10)[tl]{ }}
     
\put(340,-240){\framebox(10,10)[tl]{ }}
\put(350,-240){\framebox(10,10)[tl]{ }}
\put(360,-240){\framebox(10,10)[tl]{ }}
\put(370,-240){\framebox(10,10)[tl]{ }}
 
\put(350,-230){\framebox(10,10)[tl]{ }}
\put(360,-230){\framebox(10,10)[tl]{ }}
\put(370,-230){\framebox(10,10)[tl]{ }}

\put(360,-220){\framebox(10,10)[tl]{ }}
\put(370,-220){\framebox(10,10)[tl]{ }}

\put(370,-210){\framebox(10,10)[tl]{ }}

\put(310,-180){\framebox(10,10)[tl]{ }}
\put(310,-170){\framebox(10,10)[tl]{ }}
\put(310,-160){\framebox(10,10)[tl]{ }}
\put(320,-160){\framebox(10,10)[tl]{ }}
\put(330,-160){\framebox(10,10)[tl]{ }}
\put(340,-160){\framebox(10,10)[tl]{ }}
\put(350,-160){\framebox(10,10)[tl]{ }}

\put(310,-150){\framebox(10,10)[tl]{ }}
\put(320,-150){\framebox(10,10)[tl]{ }}
\put(330,-150){\framebox(10,10)[tl]{ }}
\put(340,-150){\framebox(10,10)[tl]{ }}
\put(350,-150){\framebox(10,10)[tl]{ }}
\put(360,-150){\framebox(10,10)[tl]{ }}
\put(370,-150){\framebox(10,10)[tl]{ }}

\put(320,-140){\framebox(10,10)[tl]{ }}
\put(330,-140){\framebox(10,10)[tl]{ }}
\put(340,-140){\framebox(10,10)[tl]{ }}
\put(350,-140){\framebox(10,10)[tl]{ }}
\put(360,-140){\framebox(10,10)[tl]{ }}
\put(370,-140){\framebox(10,10)[tl]{ }}

\put(330,-130){\framebox(10,10)[tl]{ }}
\put(340,-130){\framebox(10,10)[tl]{ }}
\put(350,-130){\framebox(10,10)[tl]{ }}
\put(360,-130){\framebox(10,10)[tl]{ }}
\put(370,-130){\framebox(10,10)[tl]{ }}
     
\put(340,-120){\framebox(10,10)[tl]{ }}
\put(350,-120){\framebox(10,10)[tl]{ }}
\put(360,-120){\framebox(10,10)[tl]{ }}
\put(370,-120){\framebox(10,10)[tl]{ }}
 
\put(350,-110){\framebox(10,10)[tl]{ }}
\put(360,-110){\framebox(10,10)[tl]{ }}
\put(370,-110){\framebox(10,10)[tl]{ }}

\put(360,-100){\framebox(10,10)[tl]{ }}
\put(370,-100){\framebox(10,10)[tl]{ }}

\put(370,-90){\framebox(10,10)[tl]{ }}

\put(310,-50){\framebox(10,10)[tl]{ }}
\put(310,-40){\framebox(10,10)[tl]{ }}
\put(310,-30){\framebox(10,10)[tl]{ }}
\put(310,-20){\framebox(10,10)[tl]{ }}
\put(320,-20){\framebox(10,10)[tl]{ }}
\put(330,-20){\framebox(10,10)[tl]{ }}
\put(340,-20){\framebox(10,10)[tl]{ }}

\put(310,-10){\framebox(10,10)[tl]{ }}
\put(320,-10){\framebox(10,10)[tl]{ }}
\put(330,-10){\framebox(10,10)[tl]{ }}
\put(340,-10){\framebox(10,10)[tl]{ }}
\put(350,-10){\framebox(10,10)[tl]{ }}
\put(360,-10){\framebox(10,10)[tl]{ }}
\put(370,-10){\framebox(10,10)[tl]{ }}

\put(320,0){\framebox(10,10)[tl]{ }}
\put(330,0){\framebox(10,10)[tl]{ }}
\put(340,0){\framebox(10,10)[tl]{ }}
\put(350,0){\framebox(10,10)[tl]{ }}
\put(360,0){\framebox(10,10)[tl]{ }}
\put(370,0){\framebox(10,10)[tl]{ }}

\put(330,10){\framebox(10,10)[tl]{ }}
\put(340,10){\framebox(10,10)[tl]{ }}
\put(350,10){\framebox(10,10)[tl]{ }}
\put(360,10){\framebox(10,10)[tl]{ }}
\put(370,10){\framebox(10,10)[tl]{ }}
     
\put(340,20){\framebox(10,10)[tl]{ }}
\put(350,20){\framebox(10,10)[tl]{ }}
\put(360,20){\framebox(10,10)[tl]{ }}
\put(370,20){\framebox(10,10)[tl]{ }}
 
\put(350,30){\framebox(10,10)[tl]{ }}
\put(360,30){\framebox(10,10)[tl]{ }}
\put(370,30){\framebox(10,10)[tl]{ }}

\put(360,40){\framebox(10,10)[tl]{ }}
\put(370,40){\framebox(10,10)[tl]{ }}

\put(370,50){\framebox(10,10)[tl]{ }}

\put(200,-60){\framebox(10,10)[tl]{ }}
\put(200,-50){\framebox(10,10)[tl]{ }}
\put(200,-40){\framebox(10,10)[tl]{ }}
\put(200,-30){\framebox(10,10)[tl]{ }}
\put(200,-20){\framebox(10,10)[tl]{ }}
\put(210,-20){\framebox(10,10)[tl]{ }}
\put(220,-20){\framebox(10,10)[tl]{ }}

\put(200,-10){\framebox(10,10)[tl]{ }}
\put(210,-10){\framebox(10,10)[tl]{ }}
\put(220,-10){\framebox(10,10)[tl]{ }}
\put(230,-10){\framebox(10,10)[tl]{ }}
\put(240,-10){\framebox(10,10)[tl]{ }}
\put(250,-10){\framebox(10,10)[tl]{ }}
\put(260,-10){\framebox(10,10)[tl]{ }}

\put(210,0){\framebox(10,10)[tl]{ }}
\put(220,0){\framebox(10,10)[tl]{ }}
\put(230,0){\framebox(10,10)[tl]{ }}
\put(240,0){\framebox(10,10)[tl]{ }}
\put(250,0){\framebox(10,10)[tl]{ }}
\put(260,0){\framebox(10,10)[tl]{ }}

\put(220,10){\framebox(10,10)[tl]{ }}
\put(230,10){\framebox(10,10)[tl]{ }}
\put(240,10){\framebox(10,10)[tl]{ }}
\put(250,10){\framebox(10,10)[tl]{ }}
\put(260,10){\framebox(10,10)[tl]{ }}
     
\put(230,20){\framebox(10,10)[tl]{ }}
\put(240,20){\framebox(10,10)[tl]{ }}
\put(250,20){\framebox(10,10)[tl]{ }}
\put(260,20){\framebox(10,10)[tl]{ }}
 
\put(240,30){\framebox(10,10)[tl]{ }}
\put(250,30){\framebox(10,10)[tl]{ }}
\put(260,30){\framebox(10,10)[tl]{ }}

\put(250,40){\framebox(10,10)[tl]{ }}
\put(260,40){\framebox(10,10)[tl]{ }}

\put(260,50){\framebox(10,10)[tl]{ }}

\put(90,-70){\framebox(10,10)[tl]{ }}
\put(90,-60){\framebox(10,10)[tl]{ }}
\put(90,-50){\framebox(10,10)[tl]{ }}
\put(90,-40){\framebox(10,10)[tl]{ }}
\put(90,-30){\framebox(10,10)[tl]{ }}
\put(90,-20){\framebox(10,10)[tl]{ }}
\put(100,-20){\framebox(10,10)[tl]{ }}

\put(90,-10){\framebox(10,10)[tl]{ }}
\put(100,-10){\framebox(10,10)[tl]{ }}
\put(110,-10){\framebox(10,10)[tl]{ }}
\put(120,-10){\framebox(10,10)[tl]{ }}
\put(130,-10){\framebox(10,10)[tl]{ }}
\put(140,-10){\framebox(10,10)[tl]{ }}
\put(150,-10){\framebox(10,10)[tl]{ }}

\put(100,0){\framebox(10,10)[tl]{ }}
\put(110,0){\framebox(10,10)[tl]{ }}
\put(120,0){\framebox(10,10)[tl]{ }}
\put(130,0){\framebox(10,10)[tl]{ }}
\put(140,0){\framebox(10,10)[tl]{ }}
\put(150,0){\framebox(10,10)[tl]{ }}

\put(110,10){\framebox(10,10)[tl]{ }}
\put(120,10){\framebox(10,10)[tl]{ }}
\put(130,10){\framebox(10,10)[tl]{ }}
\put(140,10){\framebox(10,10)[tl]{ }}
\put(150,10){\framebox(10,10)[tl]{ }}
     
\put(120,20){\framebox(10,10)[tl]{ }}
\put(130,20){\framebox(10,10)[tl]{ }}
\put(140,20){\framebox(10,10)[tl]{ }}
\put(150,20){\framebox(10,10)[tl]{ }}
 
\put(130,30){\framebox(10,10)[tl]{ }}
\put(140,30){\framebox(10,10)[tl]{ }}
\put(150,30){\framebox(10,10)[tl]{ }}

\put(140,40){\framebox(10,10)[tl]{ }}
\put(150,40){\framebox(10,10)[tl]{ }}

\put(150,50){\framebox(10,10)[tl]{ }}

\put(-20,-70){\framebox(10,10)[tl]{ }}
\put(-20,-60){\framebox(10,10)[tl]{ }}
\put(-20,-50){\framebox(10,10)[tl]{ }}
\put(-20,-40){\framebox(10,10)[tl]{ }}
\put(-20,-30){\framebox(10,10)[tl]{ }}
\put(-20,-20){\framebox(10,10)[tl]{ }}
\put(-20,-80){\framebox(10,10)[tl]{ }}

\put(-20,-10){\framebox(10,10)[tl]{ }}
\put(-10,-10){\framebox(10,10)[tl]{ }}
\put(0,-10){\framebox(10,10)[tl]{ }}
\put(10,-10){\framebox(10,10)[tl]{ }}
\put(20,-10){\framebox(10,10)[tl]{ }}
\put(30,-10){\framebox(10,10)[tl]{ }}
\put(40,-10){\framebox(10,10)[tl]{ }}

\put(-10,0){\framebox(10,10)[tl]{ }}
\put(0,0){\framebox(10,10)[tl]{ }}
\put(10,0){\framebox(10,10)[tl]{ }}
\put(20,0){\framebox(10,10)[tl]{ }}
\put(30,0){\framebox(10,10)[tl]{ }}
\put(40,0){\framebox(10,10)[tl]{ }}

\put(0,10){\framebox(10,10)[tl]{ }}
\put(10,10){\framebox(10,10)[tl]{ }}
\put(20,10){\framebox(10,10)[tl]{ }}
\put(30,10){\framebox(10,10)[tl]{ }}
\put(40,10){\framebox(10,10)[tl]{ }}
     
\put(10,20){\framebox(10,10)[tl]{ }}
\put(20,20){\framebox(10,10)[tl]{ }}
\put(30,20){\framebox(10,10)[tl]{ }}
\put(40,20){\framebox(10,10)[tl]{ }}
 
\put(20,30){\framebox(10,10)[tl]{ }}
\put(30,30){\framebox(10,10)[tl]{ }}
\put(40,30){\framebox(10,10)[tl]{ }}

\put(30,40){\framebox(10,10)[tl]{ }}
\put(40,40){\framebox(10,10)[tl]{ }}

\put(40,50){\framebox(10,10)[tl]{ }}

\end{picture}

\end{example}

\

We shall summarize all the $\succeq_s$ relationships between diagrams of the form $\mathcal{S}(\lambda,\alpha;k)$ when $\lambda$ is a hook of fixed size $h \leq n+k$ and $0 \leq k \leq h$.
The restriction $h \leq n+k$ is needed to guarantee that $\mathcal{S}(\lambda,\alpha;k)$ is a skew diagram for every hook $\lambda$ of size $h$.
We impose the restriction $k \leq h$ since, for $k \geq h$, every diagram of the form $\mathcal{S}(\lambda,\alpha;k)$, when $\lambda$ is a hook of fixed size $h$, is disconnected. 
Thus, for each $k \geq h$, the skew Schur function of these diagrams factor as 
\[ s_{\mathcal{S}(\lambda,\alpha;k)} = s_{\lambda} s_{\Delta_\alpha}. \]
In particular
\[s_{\mathcal{S}(\lambda,\alpha;k)} = s_{\mathcal{S}(\lambda,\alpha;h)} \]
for all $k \geq h$, so there is no change among the differences for $k \geq h$.
Furthermore, we shall see that for $k \geq h$, none of the differences is Schur-positive.

First we give one more example. This time with $k$ varying from $0$ to $h$.

\begin{example} For each $0 \leq k \leq 6$ we show the Hasse diagrams for $\succeq_s$ on the collection of staircases with bad foundations $\mathcal{S}(\lambda,\alpha;k)$ for some $\alpha=(\alpha_1,\ldots,\alpha_n)$ where $n\geq 6$, and $\lambda$ varying over all hooks of size $h=6$.

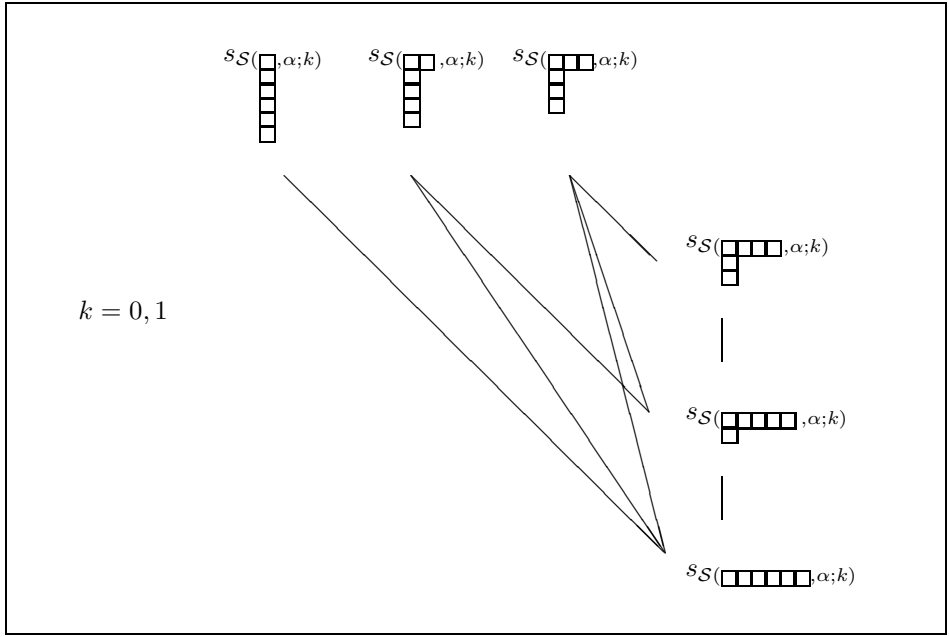
\setlength{\unitlength}{0.19mm}
\begin{picture}(100,60)(-120,0)

\put(-150,-400){\framebox(650,440)[tl]{ }}

\put(-100,-180){$k=0,1$}

\put(42,-80){\line(1,-1){263}}
\put(130,-80){\line(1,-1){165}}
\put(130,-80){\line(2,-3){176}}
\put(240,-80){\line(1,-1){60}}
\put(240,-80){\line(1,-3){55}}
\put(240,-80){\line(1,-4){66}}

\put(345,-320){\line(0,1){30}}
\put(345,-210){\line(0,1){30}}

\put(0,0){$s_{\mathcal{S}(\hooka \textrm{ } \textrm{ },\alpha;k) }$}

\put(100,0){$s_{\mathcal{S}(\hookb \textrm{ } \textrm{ } \textrm{ }  \textrm{ },\alpha;k) }$}

\put(200,0){$s_{\mathcal{S}(\hookc \textrm{ } \textrm{ }  \textrm{ } \textrm{ } \textrm{ },\alpha;k) }$}

\put(320,-130){$s_{\mathcal{S}(\hookd \textrm{ } \textrm{ } \textrm{ } \textrm{ }  \textrm{ } \textrm{ } \textrm{ },\alpha;k) }$}

\put(320,-250){$s_{\mathcal{S}(\hooke  \textrm{ } \textrm{ } \textrm{ } \textrm{ } \textrm{ } \textrm{ }  \textrm{ } \textrm{ } \textrm{ },\alpha;k) }$}

\put(320,-360){$s_{\mathcal{S}(\hookf   \textrm{ } \textrm{ } \textrm{ }\textrm{ } \textrm{ } \textrm{ } \textrm{ }  \textrm{ } \textrm{ } \textrm{ },\alpha;k) }$}

\end{picture}

\

\

\

\

\

\

\

\

\

\

\

\
\

\

\

\

\

\

\

\

\

\newpage

\setlength{\unitlength}{0.20mm}
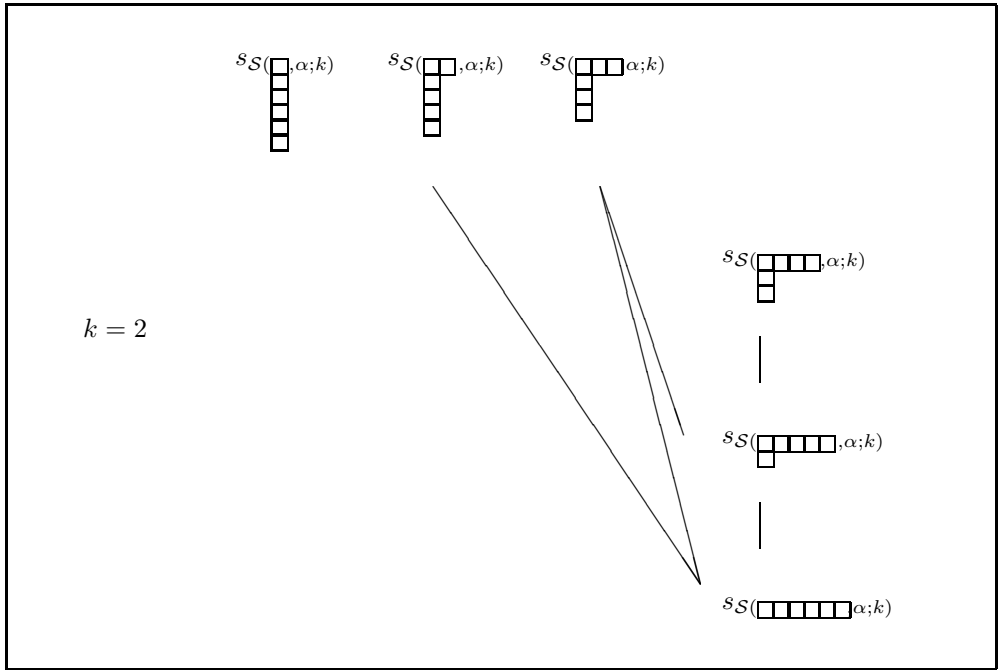
\begin{picture}(100,60)(-120,0)

\put(-100,-180){$k=2$}

\put(-150,-400){\framebox(650,440)[tl]{ }}

\put(130,-80){\line(2,-3){176}}
\put(240,-80){\line(1,-3){55}}
\put(240,-80){\line(1,-4){66}}

\put(345,-320){\line(0,1){30}}
\put(345,-210){\line(0,1){30}}

\put(0,0){$s_{\mathcal{S}(\hooka \textrm{ } \textrm{ },\alpha;k) }$}

\put(100,0){$s_{\mathcal{S}(\hookb \textrm{ }\textrm{ } \textrm{ }  \textrm{ },\alpha;k) }$}

\put(200,0){$s_{\mathcal{S}(\hookc \textrm{ } \textrm{ }  \textrm{ } \textrm{ } \textrm{ },\alpha;k) }$}

\put(320,-130){$s_{\mathcal{S}(\hookd \textrm{ } \textrm{ } \textrm{ } \textrm{ }  \textrm{ } \textrm{ } \textrm{ },\alpha;k) }$}

\put(320,-250){$s_{\mathcal{S}(\hooke  \textrm{ } \textrm{ }\textrm{ } \textrm{ } \textrm{ } \textrm{ }  \textrm{ } \textrm{ } \textrm{ },\alpha;k) }$}

\put(320,-360){$s_{\mathcal{S}(\hookf   \textrm{ } \textrm{ }\textrm{ } \textrm{ } \textrm{ } \textrm{ } \textrm{ }  \textrm{ } \textrm{ } \textrm{ },\alpha;k) }$}

\end{picture}

\

\

\

\

\

\

\

\

\

\

\

\
\

\

\

\

\

\

\

\

\

\setlength{\unitlength}{0.20mm}
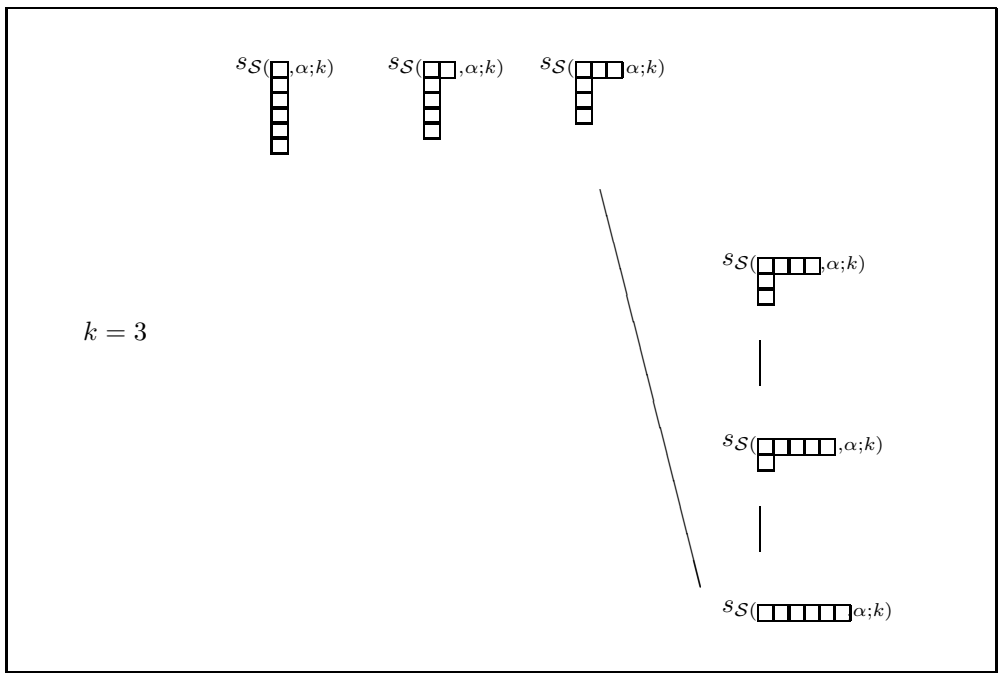
\begin{picture}(100,60)(-120,0)

\put(-100,-180){$k=3$}

\put(-150,-400){\framebox(650,440)[tl]{ }}

\put(240,-80){\line(1,-4){66}}

\put(345,-320){\line(0,1){30}}
\put(345,-210){\line(0,1){30}}

\put(0,0){$s_{\mathcal{S}(\hooka \textrm{ } \textrm{ },\alpha;k) }$}

\put(100,0){$s_{\mathcal{S}(\hookb \textrm{ }\textrm{ } \textrm{ }  \textrm{ },\alpha;k) }$}

\put(200,0){$s_{\mathcal{S}(\hookc \textrm{ } \textrm{ }  \textrm{ } \textrm{ } \textrm{ },\alpha;k) }$}

\put(320,-130){$s_{\mathcal{S}(\hookd \textrm{ } \textrm{ } \textrm{ } \textrm{ }  \textrm{ } \textrm{ } \textrm{ },\alpha;k) }$}

\put(320,-250){$s_{\mathcal{S}(\hooke  \textrm{ } \textrm{ } \textrm{ }\textrm{ } \textrm{ } \textrm{ }  \textrm{ } \textrm{ } \textrm{ },\alpha;k) }$}

\put(320,-360){$s_{\mathcal{S}(\hookf   \textrm{ } \textrm{ } \textrm{ }\textrm{ } \textrm{ } \textrm{ } \textrm{ }  \textrm{ } \textrm{ } \textrm{ },\alpha;k) }$}

\end{picture}

\

\

\

\

\

\
\

\

\

\

\

\

\

\

\

\

\

\

\

\

\
\newpage

\setlength{\unitlength}{0.20mm}
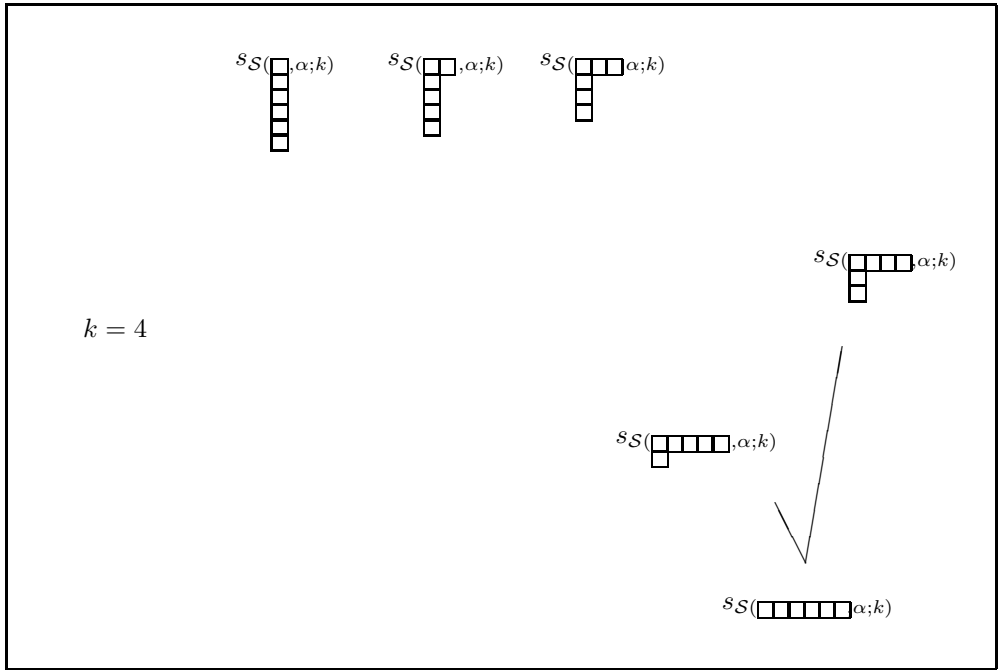
\begin{picture}(100,60)(-120,0)

\put(-100,-180){$k=4$}

\put(-150,-400){\framebox(650,440)[tl]{ }}

\put(375,-330){\line(-1,2){20}}
\put(375,-330){\line(1,6){24}}

\put(0,0){$s_{\mathcal{S}(\hooka \textrm{ } \textrm{ },\alpha;k) }$}

\put(100,0){$s_{\mathcal{S}(\hookb \textrm{ } \textrm{ }\textrm{ }  \textrm{ },\alpha;k) }$}

\put(200,0){$s_{\mathcal{S}(\hookc \textrm{ } \textrm{ }  \textrm{ } \textrm{ } \textrm{ },\alpha;k) }$}

\put(380,-130){$s_{\mathcal{S}(\hookd \textrm{ } \textrm{ } \textrm{ } \textrm{ }  \textrm{ } \textrm{ } \textrm{ },\alpha;k) }$}

\put(250,-250){$s_{\mathcal{S}(\hooke  \textrm{ } \textrm{ }\textrm{ } \textrm{ } \textrm{ } \textrm{ }  \textrm{ } \textrm{ } \textrm{ },\alpha;k) }$}

\put(320,-360){$s_{\mathcal{S}(\hookf   \textrm{ } \textrm{ }\textrm{ } \textrm{ } \textrm{ } \textrm{ } \textrm{ }  \textrm{ } \textrm{ } \textrm{ },\alpha;k) }$}

\end{picture}

\

\

\

\

\

\

\

\

\

\

\

\
\

\

\

\

\

\

\

\

\

\setlength{\unitlength}{0.20mm}
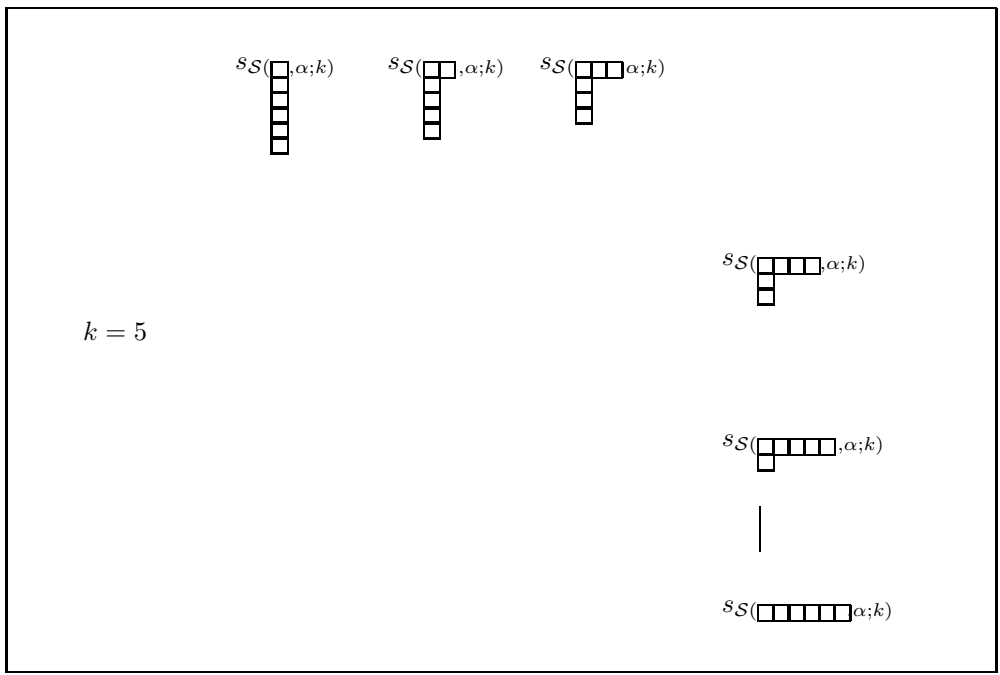
\begin{picture}(100,60)(-120,0)

\put(-100,-180){$k=5$}

\put(-150,-400){\framebox(650,440)[tl]{ }}

\put(345,-320){\line(0,1){30}}

\put(0,0){$s_{\mathcal{S}(\hooka \textrm{ } \textrm{ },\alpha;k) }$}

\put(100,0){$s_{\mathcal{S}(\hookb \textrm{ }\textrm{ } \textrm{ }  \textrm{ },\alpha;k) }$}

\put(200,0){$s_{\mathcal{S}(\hookc \textrm{ } \textrm{ }  \textrm{ } \textrm{ } \textrm{ },\alpha;k) }$}

\put(320,-130){$s_{\mathcal{S}(\hookd \textrm{ } \textrm{ } \textrm{ } \textrm{ }  \textrm{ } \textrm{ } \textrm{ },\alpha;k) }$}

\put(320,-250){$s_{\mathcal{S}(\hooke  \textrm{ }\textrm{ } \textrm{ } \textrm{ } \textrm{ } \textrm{ }  \textrm{ } \textrm{ } \textrm{ },\alpha;k) }$}

\put(320,-360){$s_{\mathcal{S}(\hookf   \textrm{ }\textrm{ } \textrm{ } \textrm{ } \textrm{ } \textrm{ } \textrm{ }  \textrm{ } \textrm{ } \textrm{ },\alpha;k) }$}

\end{picture}

\

\

\

\

\
\

\

\

\

\

\

\

\

\

\

\

\

\

\

\

\

\newpage

\setlength{\unitlength}{0.20mm}
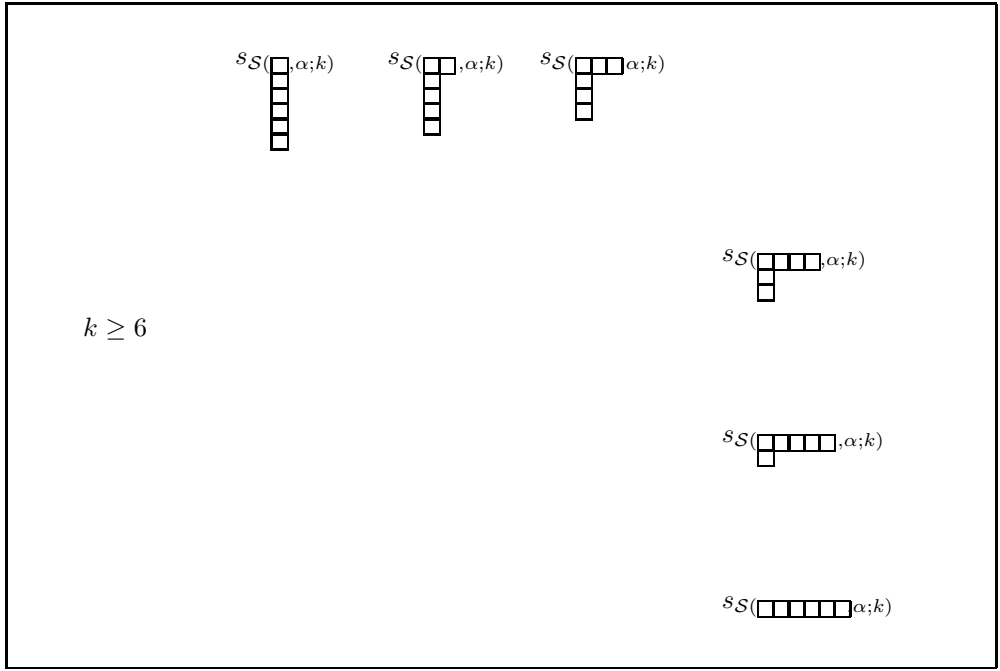
\begin{picture}(100,60)(-120,0)

\put(-100,-180){$k \geq 6$}

\put(-150,-400){\framebox(650,440)[tl]{ }}

\put(0,0){$s_{\mathcal{S}(\hooka \textrm{ } \textrm{ },\alpha;k) }$}

\put(100,0){$s_{\mathcal{S}(\hookb \textrm{ }\textrm{ } \textrm{ }  \textrm{ },\alpha;k) }$}

\put(200,0){$s_{\mathcal{S}(\hookc \textrm{ } \textrm{ }  \textrm{ } \textrm{ } \textrm{ },\alpha;k) }$}

\put(320,-130){$s_{\mathcal{S}(\hookd \textrm{ } \textrm{ } \textrm{ } \textrm{ }  \textrm{ } \textrm{ } \textrm{ },\alpha;k) }$}

\put(320,-250){$s_{\mathcal{S}(\hooke  \textrm{ }\textrm{ } \textrm{ } \textrm{ } \textrm{ } \textrm{ }  \textrm{ } \textrm{ } \textrm{ },\alpha;k) }$}

\put(320,-360){$s_{\mathcal{S}(\hookf   \textrm{ }\textrm{ } \textrm{ } \textrm{ } \textrm{ } \textrm{ } \textrm{ }  \textrm{ } \textrm{ } \textrm{ },\alpha;k) }$}

\end{picture}

\

\

\

\

\

\

\

\

\

\

\

\

\

\

\

\

\

\

\

\

\

For $0 \leq k \leq 1$ we have the same structure of the Hasse diagrams that was displayed in the first example.  
As $k$ increases, fewer of the $\succeq_s$ relations remain satisfied, until finally, when $k \geq 6$, there are no Schur-positive differences among these diagrams. 
We note that the chain that was apparent among the diagrams on the right when $k=0$ also lost its structure as $k$ increased. 

\end{example}

\

When working with hooks, we shall find it convenient to describe each hook by its \textit{arm length} and \textit{leg length}.
Hence, we let $\mu$ be the hook $(\mu_a,1^{\mu_l - 1})$ and $\lambda$ be the hook $(\lambda_a,1^{\lambda_l - 1})$. 
Throughout this section we shall use a fixed fat staircase $\Delta_\alpha$, where $\alpha = (\alpha_1,\alpha_2,\ldots,\alpha_n)$. 
Thus $n$ is the width of the fat staircase.

The following results summarize all the $\succeq_s$ relationships between diagrams of the form $\mathcal{S}(\lambda,\alpha;k)$ when $\lambda$ is a hook of fixed size $h \leq n+k$ and $0 \leq k \leq h$.

For each pair of hooks $\lambda$, $\mu$ with $\lambda_a, \mu_a \leq \longlroof \frac{h}{2} \longrroof$, Theorem~\ref{kantichainhooks1} and Theorem~\ref{kantichainhooks11} each prove one side of the Schur-incomparability of this pair, thus describing the antichain structure displayed along the top of the Hasse diagrams in the previous examples. 

For each pair of hooks $\lambda$, $\mu$ with $\longlroof \frac{h}{2} \longrroof \leq \lambda_a < \mu_a$, Theorem~\ref{kchainhooks} and Theorem~\ref{kchainhooksb} shows that $\mathcal{S}(\lambda,\alpha;k) \succeq_s \mathcal{S}(\mu,\alpha;k)$ if and only if $\lambda_a \geq \mu_l+k-1$ . 
This describes relations among those diagrams displayed along the right of the Hasse diagrams in the previous examples. 

Finally, for each pair of hooks $\lambda$, $\mu$ with $\lambda_a, \mu_l < \longlroof \frac{h}{2} \longrroof$, 
Theorem~\ref{kcrosshooks1} and Theorem~\ref{kcrosshooks11} shows that when $1 \leq k \leq h$ we have $\lambda_a \geq \mu_l+k-1$ if and only if $\mathcal{S}(\lambda, \alpha;k) \succeq_s \mathcal{S}(\mu, \alpha;k)$, and when $k=0$ we have $\lambda_a \geq \mu_l$ if and only if $\mathcal{S}(\lambda, \alpha;k) \succeq_s \mathcal{S}(\mu, \alpha;k)$. 
This describes the relationships between the diagrams displayed on the right with the diagrams displayed along the top in the previous Hasse diagrams.

\

Let us finally begin. We start by looking at the antichain structure.

\begin{theorem}
\label{kantichainhooks1}
Let $\lambda$ and $\mu$ be distinct hooks with $|\lambda|=|\mu| =h \leq n+k$ and $\lambda_a < \mu_a \leq \longlroof \frac{h}{2} \longrroof$, and let $0 \leq k \leq h$. 
Then $\mathcal{S}(\mu, \alpha;k) \not\succeq_s  \mathcal{S}(\lambda, \alpha;k)$.
\end{theorem}
\begin{proof}
We shall show that there exists a SSYT $\mathcal{T}$ of shape $\mathcal{S}(\lambda, \alpha;k)$ with lattice reading word such that there is no SSYT of shape $\mathcal{S}(\mu, \alpha;k)$ with lattice reading word having the same content.
This is sufficient to prove the theorem.

Since $\lambda_a < \mu_a$ and $|\lambda|=|\mu|$, we have $\lambda_l > \mu_l$.
Let $r_1=r_2= \ldots = r_k=1$ and let $r_{1+k} < r_{2+k} < \ldots < r_{n+k}$ be the values of $R_{\alpha,k}$ greater than $1$.
We can create a SSYT of shape $\lambda$ by filling the boxes of $\lambda$ as follows.

\begin{center}
\begin{tabular}{cccccccc}
$r_1$ & \ldots & $r_k$ & $r_{1+k}$ & $r_{2+k}$ & $\cdots$ & $r_{\lambda_a-1}$ & $|\alpha|+1$ \\
$|\alpha|+2$\\
$|\alpha|+3$\\
$\vdots$\\
$|\alpha|+\lambda_l$ \\
\end{tabular}
\end{center}

\noindent Using the unique filling of $\Delta_\alpha$, it is easy to check that the resulting tableau of shape $\lambda \oplus \Delta_\alpha$ has lattice reading word since each of the entries in the first row of $\lambda$ are from $R_{\alpha,k}$ and the entry $1$ appears $k$ times.
Thus Lemma~\ref{kfatjoinlemma} provides us with a SSYT $\mathcal{T}$ of shape $\mathcal{S}(\lambda, \alpha;k)$ with lattice reading word, where $\lambda$ is filled as shown above.

Since $\mu_l < \lambda_l$, we have $l(\mathcal{S}(\mu, \alpha;k))=|\alpha|+\mu_l<|\alpha|+\lambda_l$. 
Therefore no SSYT of shape $\mathcal{S}(\mu, \alpha;k)$  with lattice reading word can contain the entry $|\alpha|+\lambda_l$.
Thus no SSYT of shape $\mathcal{S}(\mu, \alpha;k)$ with lattice reading word can have the same content as $\mathcal{T}$.
Hence, it follows that $s_{\mathcal{S}(\mu,\alpha;k)}-s_{\mathcal{S}(\lambda,\alpha;k)} \not\geq_s 0$. \qed
\end{proof}

\begin{theorem}
\label{kantichainhooks11}
Let $\lambda$ and $\mu$ be distinct hooks with $|\lambda|=|\mu| =h \leq n+k$ and $\lambda_a < \mu_a \leq \longlroof \frac{h}{2} \longrroof$ and let $0 \leq k \leq h$. 
Then $\mathcal{S}(\lambda, \alpha;k) \not\succeq_s  \mathcal{S}(\mu, \alpha;k)$.
\end{theorem}
\begin{proof}
We first consider the case when $k>0$. 
We let $r_1 = r_2 = \ldots = r_k = 1$ and let $r_{1+k} < r_{2+k} < \ldots r_{n+k}$ be the values of $R_{\alpha,k}$ greater than 1, where we note $n+k \geq h$.
We can create a SSYT of shape $\mu$ by filling the boxes of $\mu$ as follows.

\begin{center}
\begin{tabular}{cccccccc}
$r_1$ & $\cdots$ & $r_k$ & $r_{1+k}$ &  $r_{2+k}$ & $\cdots$ & $r_{\mu_a}$ \\
$r_{\mu_a+1}$\\
$r_{\mu_a+2}$\\
$\vdots$\\
$r_{h}$ \\
\end{tabular}
\end{center}

\noindent Since $r_1 = r_2 = \ldots = r_k = 1$ and $r_{1+k} < r_{2+k} < \ldots < r_{h}$ are distinct values of $R_{\alpha,k}$, it is easy to check that the resulting tableau of shape $\mu \oplus \Delta_\alpha$ has lattice reading word. 
Thus Lemma~\ref{kfatjoinlemma} provides us with a SSYT $\mathcal{T}$ of shape $\mathcal{S}(\mu, \alpha;k)$ with lattice reading word, where $\mu$ is filled as shown above.

We now wish to count all SSYTx of shape $\mathcal{S}(\mu, \alpha;k)$ (shape $\mathcal{S}(\lambda, \alpha;k)$, respectively) with content $\nu=c(\mathcal{T})$. 
Since $\Delta_\alpha$ has a unique way of being filled, we must find all semistandard fillings of $\mu$ ($\lambda$, resp.) with the values $r_1,r_2,\ldots,r_h$. 
Since $r_1 = r_2 = \ldots = r_k = 1$ and $r_{1+k} < r_{2+k} < \ldots < r_{h}$, the values $r_1, r_2, \ldots, r_k$ must appear as the first $k$ values of the first row of $\mu$ ($\lambda$, resp.). 
Further, once we choose $\mu_a-k$ ($\lambda_a -k$, resp.) of the values $r_{1+k}, r_{2+k}, \ldots, r_h$ to appear in the first row of $\mu$ (first row of $\lambda$, resp.), then the remaining $r$'s must appear in the first column and the order of all these values is uniquely determined by the semistandard conditions.

Therefore the number of SSYTx of shape $\mathcal{S}(\mu, \alpha ;k)= \kappa' / \rho'$ with lattice reading word and content $\nu =c(\mathcal{T})$ is given by
\[ c_{\rho' \nu}^{\kappa'} = \left( \begin{array}{c}
                  h-k \\
                    \mu_a -k\\
           \end{array} \right) \]
and the number of SSYTx of shape $\mathcal{S}(\lambda, \alpha ;k)= \kappa / \rho$ with lattice reading word and content $\nu=c(\mathcal{T})$ is given by
\[ c_{\rho \nu}^{\kappa} = \left( \begin{array}{c}
                  h-k \\
                    \lambda_a -k\\
           \end{array} \right).  \]

\

Since $\lambda_a < \mu_a \leq \longlroof \frac{h}{2} \longrroof$, we have $\lambda_a +\mu_a < h+1 \leq h+k$.
Therefore we have $h - \lambda_a > \mu_a-k$ and we obtain
\begin{equation}
\label{antieq1}
h - \lambda_a -i> \mu_a-k-i,
\end{equation}
for each $i$.

\

\noindent Therefore
\begin{eqnarray*}
\  c_{\rho' \nu}^{\kappa'} & = & \frac{(h-k)! }{(h-\mu_a)!(\mu_a-k)!}
\\ & = & \frac{(h-k)! }{(h-\lambda_a)!(\lambda_a-k)!} \times \prod_{i=0}^{\mu_a - \lambda_a-1} \frac{h-\lambda_a-i}{\mu_a-k -i}
\\ & = & c_{\rho \nu}^{\kappa} \times \prod_{i=0}^{\mu_a - \lambda_a-1} \frac{h-\lambda_a-i}{\mu_a-k -i}
\\ & > & c_{\rho \nu}^{\kappa} ,
\end{eqnarray*}
where we have used Equation~\ref{antieq1} in the final step. Therefore $s_{\mathcal{S}(\lambda,\alpha ;k)}-s_{\mathcal{S}(\mu,\alpha ;k)} \not\geq_s 0$.

\

Now consider the case $k=0$. 
We now let $r_{1} < r_{2} < \ldots < r_{n}$ be the values of $R_{\alpha,k}$.
We can create a SSYT of shape $\mu$ by filling the boxes of $\mu$ as follows.

\begin{center}
\begin{tabular}{cccccccccc}
$r_1$ & $r_2$ & $\cdots$ &  $r_{\mu_a}$ \\
$r_{\mu_a+1}$\\
$r_{\mu_a+2}$\\
$\vdots$\\
$r_{h}$ \\
\end{tabular}
\end{center}

As before, Lemma~\ref{kfatjoinlemma} provides us with a SSYT $\mathcal{T}$ of shape $\mathcal{S}(\mu, \alpha;k)$ with lattice reading word, where $\mu$ is filled as shown above.
We now wish to count all SSYTx of shape $\mathcal{S}(\mu, \alpha;k)$ (shape $\mathcal{S}(\lambda, \alpha;k)$, respectively) with content $\nu =c(\mathcal{T})$.
In this case only $r_1$ is required to appear at the beginning of the first row of $\mu$ ($\lambda$, resp.). 
Further, once we choose $\mu_a-1$ ($\lambda_a -1$, resp.) of the values $r_{2} < \ldots <r_h$ to appear in the first row of $\mu$ (first row of $\lambda$, resp.), then the remaining $r$'s must appear in the first column of $\mu$ (first column of $\lambda$, resp.) and the order of all these values is uniquely determined by the semistandard conditions.

Therefore the number of SSYTx of shape $\mathcal{S}(\mu, \alpha ;k)= \kappa' / \rho'$ with lattice reading word and content $\nu =c(\mathcal{T})$ is given by
\[ c_{\rho' \nu}^{\kappa'} = \left( \begin{array}{c}
                  h-1 \\
                    \mu_a -1\\
           \end{array} \right) \]
and the number of SSYTx of shape $\mathcal{S}(\lambda, \alpha ;k)= \kappa / \rho$ with lattice reading word and content $\nu=c(\mathcal{T})$ is given by
\[ c_{\rho \nu}^{\kappa} = \left( \begin{array}{c}
                  h-1 \\
                    \lambda_a -1\\
           \end{array} \right).  \]

\

Since $\lambda_a < \mu_a \leq \longlroof \frac{h}{2} \longrroof$, we have $h+1 > \lambda_a +\mu_a$.
Therefore we have $h - \lambda_a > \mu_a-1$ and we once again obtain
\begin{equation}
\label{antieqkkkr}
h - \lambda_a -i> \mu_a-1-i,
\end{equation}
for each $i$.

\

\noindent Therefore
\begin{eqnarray*}
\  c_{\rho' \nu}^{\kappa'} & = & \frac{(h-1)! }{(h-\mu_a)!(\mu_a-1)!}
\\ & = & \frac{(h-1)! }{(h-\lambda_a)!(\lambda_a-1)!} \times \prod_{i=0}^{\mu_a - \lambda_a -1} \frac{h-\lambda_a-i}{\mu_l-1 -i}
\\ & = & c_{\rho \nu}^{\kappa} \times \prod_{i=0}^{\mu_l - \lambda_a-1} \frac{h-\lambda_a-i}{\mu_a-1 -i}
\\ & > & c_{\rho \nu}^{\kappa} ,
\end{eqnarray*}
where we have used Equation~\ref{antieqkkkr} in the final step. Therefore $s_{\mathcal{S}(\lambda,\alpha ;k)}-s_{\mathcal{S}(\mu,\alpha ;k)} \not\geq_s 0$. \qed

\end{proof}

We now depart from looking at the hooks 
$\lambda,\mu$ satisfying $\lambda_a < \mu_a \leq \longlroof \frac{h}{2} \longrroof$. 
Instead, we turn to the hooks 
$\lambda,\mu$ satisfying $\longlroof \frac{h}{2} \longrroof \leq \lambda_a < \mu_a$.
The following theorems describes the relations among the diagrams we displayed on the right in our examples. 

\begin{theorem}
\label{kchainhooks}
Let $\lambda$ and $\mu$ be hooks with $|\lambda|=|\mu|=h \leq n+k$ and $\longlroof \frac{h}{2} \longrroof \leq \lambda_a < \mu_a$, and let $0 \leq k \leq h$.
If $\lambda_a \geq \mu_l + k -1$ then $\mathcal{S}(\lambda, \alpha ;k) \succeq_s \mathcal{S}(\mu, \alpha ;k)$.

\end{theorem}

\begin{proof}

To prove the result, we shall consider any content $\nu$ such that a SSYT of shape $\mathcal{S}(\mu,\alpha ;k)$ with content $\nu$ and lattice reading word exists.
First we shall show that there is also a SSYT of shape $\mathcal{S}(\lambda,\alpha ;k)$ with content $\nu$ and lattice reading word.
Then, letting $\mathcal{S}(\lambda,\alpha ;k)=\kappa / \rho$ and $\mathcal{S}(\mu,\alpha ;k)=\kappa' / {\rho'}$ for partitions $\kappa$, $\kappa'$, $\rho$, and $\rho'$, we shall show that the Littlewood-Richardson coefficients for these two diagrams and this content satisfy 
\[ c_{\rho \nu}^{\kappa} \geq c_{\rho' \nu}^{\kappa'}.\]
Having shown that this inequality holds for any content $\nu$ for which a SSYT of shape $\mathcal{S}(\mu,\alpha ;k)$ with content $\nu$ and lattice reading word exists, this will imply that $s_{\mathcal{S}(\lambda, \alpha ;k)}-s_{\mathcal{S}(\mu,\alpha ;k)} \geq_s 0.$

\

Let $\nu$ be a content such that there is a SSYT $\mathcal{T}_1$ of shape $\mathcal{S}(\mu,\alpha ;k)$ with content $\nu$ and lattice reading word.
By Lemma~\ref{kfatfirstrowlemma} we know that the first row of $\mu$ contains at most $k$ $1$'s. 
Let $q$ be the number of $1$'s in the first row of $\mu$. We write $a_1=a_2 = \ldots a_q =1$. 
Then the rest of the first row of $\mu$ consists of a strictly increasing sequence $a_{q+1}<a_{q+2}<\ldots < a_{\mu_a}$ where each $a_{q+i} \in R_{\alpha,k}$ with $a_{q+i}>1$. 
Also, since the columns of $\mathcal{T}_1$ strictly increase, the first column of $\lambda$ contains a strictly increasing sequence $a'_1<a'_2<\ldots < a'_{\mu_l}$, where $a'_1=a_1$.
Since $\Delta_\alpha$ can only be filled in one way (in both shapes $\mathcal{S}(\mu,\alpha ;k)$ and $\mathcal{S}(\lambda,\alpha ;k)$) and the values $a_1$, \ldots $a_q$ must be placed in the first $q$ positions in the first row of either foundation, in order to obtain a tableau of shape $\mathcal{S}(\lambda,\alpha ;k)$ and content $\nu$ we only need to show how to place the values of $\{a_{q+1},a_{q+2},\ldots,a_{\mu_a}\} \cup \{a'_2,a'_3\ldots,a'_{\mu_l}\}$ in $\lambda$.

\

If $l(\nu) > |\alpha|+1$ for this particular content $\nu$ then there are entries of $\mathcal{T}_1$ greater than $|\alpha|+1$.
Since $\Delta_\alpha$ has content $\delta_\alpha$, the lattice condition implies that $|\alpha|+2$ appears in $\mu$.
Since each $a_i \in R_{\alpha,k}$, we have $a_i \leq |\alpha|+1$ for each $i$ and so $a'_j=|\alpha|+2$ for some $j$. 
The lattice condition and the fact that the column strictly increases gives that 

\begin{eqnarray*}
a'_j &=& |\alpha|+2 \\
a'_{j+1} &=& |\alpha|+3 \\
&\vdots&\\
a'_{\mu_l} &=& |\alpha|+ \mu_l -j +2. \\
\end{eqnarray*}

\noindent Again, by the lattice condition, it is clear that any SSYT of shape $\mathcal{S}(\lambda,\alpha ;k)$ with lattice reading word and content $\nu$ must also have these values $a'_j, a'_{j+1},\ldots, a'_{\mu_l}$ as the last $\mu_l - j+1$ entries of the first column of $\lambda$. 
Since $\lambda_a < \mu_a$ gives $\mu_l < \lambda_l$, we have $\mu_l - j+1 < \lambda_l -j+1$.
Since $j \geq 2$ this gives,
\begin{equation}
\label{kneedeqn3}
\mu_l - j+1  \leq \lambda_l,
\end{equation}
so these entries do fit in this column.
Note that if $l(\nu) \leq |\alpha|+1$, then this sequence of values $a'_j, a'_{j+1},\ldots, a'_{\mu_l}$ is empty and we do not have to worry about placing any entries larger than $|\alpha|+1$ into $\lambda$. 
(In such a case we may consider $j= \mu_l +1$.)

Let $M$ be the multiset $\{a_{q+1},a_{q+2},\ldots,a_{\mu_a}\} \cup \{a'_2,a'_3\ldots,a'_{j-1} \}$. 
Then $M$ is the remaining entries that we still need to place in $\lambda$ to obtain a tableau of shape $\mathcal{S}(\lambda,\alpha;k)$ and content $\nu$. 
We have $|M|=\mu_a+j-2-q$ and $\textrm{max}(M)=|\alpha|+1$. 
Let $R =\{a_{q+1},a_{q+2},\ldots,a_{\mu_a}\} \cap \{a'_2,a'_3,\ldots,a'_{j-1} \}$.
We note that $R =\{a_{q+1},a_{q+2},\ldots,a_{\mu_a}\} \cap \{a'_2,a'_3,\ldots,a'_{\mu_l} \}$ since Lemma~\ref{kfatfirstrowlemma} shows $a_{\mu_a} \in R_{\alpha,k}$, which implies $a_{\mu_a} < |\alpha|+2 = a'_{j}$.
Thus \textit{$R$ is the set of values (except 1) that appear in both the first row and the first column of $\mu$.} 

Since the values of $R$ all appear in the first row of $\mu$, Lemma~\ref{kfatfirstrowlemma} gives that $R \subseteq R_{\alpha,k}$.
For any SSYT of shape $\mathcal{S}(\lambda,\alpha ;k)$ with lattice reading word and content $\nu$, Lemma~\ref{kfatfirstrowlemma} shows that, besides $1$, the values in the first row of $\lambda$ are distinct, so, when creating a filling of $\lambda$, the values in $R$ must also appear in both the first row of $\lambda$ and the first column of $\lambda$.

Consider $A=\{a_{q+1},a_{q+2},\ldots,a_{\mu_a}\}-R$ and $A'=\{a'_2,a'_3\ldots,a'_{j-1} \}-R$.
Since we know that the values of $R$ must appear in both the first row of $\lambda$ and first column of $\lambda$, $A \cup A'$ contains the remaining values of $M$ that need to be placed in $\lambda$. 
In other words, \textit{$A \cup A'$ is the set of all values $\leq |\alpha|+1$ that can appear in exactly one of the first row of $\lambda$ or the first column of $\lambda$.}

We wish to show that 
\begin{equation}
\label{keqneed1}
|R| \leq  \lambda_a -q
\end{equation} holds.
If $q=0$, then $|R| \leq j \leq \mu_l+1 \leq \lambda_l < \lambda_a = \lambda_a -q$. 
Now, when $q \geq 1$ the top-left entry of $\mu$ is $1$ which is not in $R$, hence we have $|R| \leq \mu_l -1 \leq \lambda_a -k \leq \lambda_a -q$, where
we have used the fact that $\lambda_a \geq \mu_l + k -1$.

Now, because Equation~\ref{keqneed1} holds,  
we can extend the values of $R$ to an increasing sequence $b_{q+1} < b_{q+2} < \ldots < b_{\lambda_a}$ by choosing $\lambda_a - |R|-q$ additional values from $(A \cup A') \cap R_{\alpha,k}$. 
There are enough values to choose from since there are 
\begin{equation}
\label{kneed4}
\mu_a - |R| -q \geq \lambda_a -|R|-q
\end{equation}
values of $(A \cup A')\cap R_{\alpha,k}$ present in the first row of $\mu$.
The sequence of $b_i$'s is strictly increasing since $(A \cup A') \cap R = \emptyset$.

Now $M-\{b_{q+1},\ldots,b_{\lambda_a} \} \subseteq M-R$ contains $w=|M|-(\lambda_a-q)=\mu_a+j-2 -\lambda_a$ distinct values, each no greater than $|\alpha|+1$.
That is, they are an increasing sequence $c_1 < c_2 < \ldots < c_{w}$, where $c_{w}\leq |\alpha|+1$ and $c_1>1$.
We have 
\begin{eqnarray*}
\  \lambda_l & = & \mu_a + \mu_l - \lambda_a
\\   & = & 1+ (\mu_a +j-2-\lambda_a) + (\mu_l -j +1)
\\ & = & 1+w+(\mu_l -j +1),
\\
\end{eqnarray*}
so, letting $b_1=b_2=\ldots=b_q=1$, we may fill $\lambda$ as shown below.

\

\begin{center}
\begin{tabular}{cccc}
$b_1$ & $b_2$ & $\cdots$ & $b_{\lambda_a}$ \\
$c_1$\\
$c_2$\\
$\vdots$\\
$c_{w}$\\
$a'_{j}$\\
$a'_{j+1}$\\
$\vdots$\\
$a'_{\mu_l}$
\end{tabular}
\end{center}

\

Since the sequence of $c_i$ is strictly increasing and since $c_1 >1$ and $c_w \leq |\alpha|+1 < |\alpha|+2 =a_j'$ we have
\[ b_1 < c_1 <c_2 \ldots < c_w < a'_j < \ldots < a'_{\mu_l}.\]
That is, the first column of $\lambda$ is increasing. 
We also have 
\[ b_1=b_2=\ldots=b_q=1 \] and
\[ b_{q+1} < b_{q+2} < \ldots < b_{\mu_a},\]
so the first row of $\lambda$ is weakly increasing with $q \leq k$ $1$'s.  
Hence this filling gives us a SSYT $T$ of shape $\lambda \oplus \Delta_\alpha$ and content $\nu$. 

We now check that $T$ has a lattice reading word so that we may apply Lemma~\ref{kfatjoinlemma} to obtain the desired tableau $\mathcal{T}_2$ of shape $\mathcal{S}(\lambda,\alpha;k)$.
Suppose that $T$ does not have a lattice reading word. 
Then, when reading the foundation $\lambda$ of $T$, we must reach a point where the lattice condition failed. 
Let $x$ be the value that, when read, caused the lattice condition to fail.
The lattice condition could not have failed when reading the first row of $\lambda$ since the lattice condition places no restriction on the number of $1$'s and the remaining values in the first row of $\lambda$ were distinct values chosen from $R_{\alpha,k}$.
Therefore $x>1$ and \textit{this} $x$ which violated the lattice condition appears somewhere in the first column of $\lambda$.
We inspect the the two cases $x \in R_{\alpha,k}$ and $x \not\in R_{\alpha,k}$.

Consider the first case, $x \in R_{\alpha,k}$.
If a value of $R_{\alpha,k}$ appears only once in the foundation then reading this value cannot violate the lattice condition.
Thus, since the lattice condition failed at this $x$, this $x$ cannot be the first time that $x$ was read in $\lambda$. 
Since the columns strictly increase, the previous $x$ must have appeared in the first row of $\lambda$ and, since the values in the first row are distinct, this is the only other $x$ in $\lambda$.
Since the content of $\lambda$ is the same as the content of $\mu$, these two $x$'s appear in $\mu$ as well. 
Using the fact that $\mathcal{T}_1$ has a lattice reading word, together with the content of $\Delta_{\alpha}$, we find that the value $x-1$ appeared in $\mu$. 
Thus the value $x-1$ also appears in $\lambda$. 
Now, since both the rows and columns of $\lambda$ must weakly increase, either the $x-1$ appears in the first column of $\lambda$ above the entry $x$, or the $x-1$ appears in the first row of $\lambda$ left of the entry $x$. 
In either case the $x-1$ is read before the second $x$ is read and the lattice condition will not fail when reading this second $x$, contrary to our assuption.

We now look at the second case, where $x \not\in R_{\alpha,k}$. 
Again, the $x$ we are interested in appears in the first column of $\lambda$. 
There cannot be a second $x$ in $\lambda$ since the column strictly increases and $x \not\in R_{\alpha,k}$ implies that no other $x$ was placed in the first row of $\lambda$.
As before, $x$ must have appeared in $\mu$ and, in particular, it also appears somewhere in the first column.
Since $\mathcal{T}_1$ has a lattice reading word we must read a sequence of values $t$, $t+1$, $t+2$, $\ldots$, $x-2$, $x-1$ in $\mu$, where $t \in R_{\alpha,k}$, before we read the $x$, and we may assume that none of the values $t+1$, $t+2$, $\ldots$, $x-2$, $x-1$ are from $R_{\alpha,k}$. 
Hence each of the values $t$, $t+1$, $\ldots$, $x-2$, $x-1$ also appear in $\lambda$.
None of the values $t+1$, $t+2$, $\ldots$, $x-1$ can appear in the first row of $\lambda$ since the first row was chosen from $R_{\alpha,k}$.
That is, each value $t+1,t+2,\ldots,x-1$ appears in the first column of $\lambda$. 
Also, since both the rows and columns of $\lambda$ must weakly increase, either the $t$ appears in the first column of $\lambda$ above the entry $t+1$, or the $t$ appears in the first row of $\lambda$.
In either case the entire sequence $t$, $t+1$, $\ldots$, $x-2$, $x-1$ is read before the $x$ is read in $\lambda$ and the lattice condition does not fail at $x$, contradicting our assumption.

Since $T$ has a lattice reading word, we can now apply Lemma~\ref{kfatjoinlemma} to obtain the SSYT $\mathcal{T}_2$ of shape $\mathcal{S}(\lambda,\alpha;k)$ with lattice reading word and content $\nu$.
Therefore from any SSYT $\mathcal{T}_1$ of shape $\mathcal{S}(\mu, \alpha ;k)$ with lattice reading word and content $\nu$ we can create a SSYT $\mathcal{T}_2$ of shape $\mathcal{S}(\lambda,\alpha ;k)$ with lattice reading word and content $\nu$.

\

Let $c_{\rho \nu}^{\kappa}$ be the number of SSYTx of shape $\mathcal{S}(\lambda, \alpha;k)$ with lattice reading word and content $\nu$, and $c_{\rho' \nu}^{\kappa'}$ be the number of SSYTx of shape $\mathcal{S}(\mu,\alpha ;k)$ with lattice reading word and content $\nu$.
We shall show that the sets $R$ and $A \cup A'$ that were described above are completely determined by the content $\nu$.
That is, without starting with a specific tableau, but only starting with the desired content of a SSYT of some fat staircase with hook foundation with lattice reading word, we show how to determine $q$, the number of $1$'s in the foundation; $R$, the set of values $\neq 1$ that must appear in both the first row and first column of the foundation; and $A \cup A'$, the set of values $\leq |\alpha| +1$ that can only appear in one of the first row or the first column of the foundation. 

Since $\Delta_\alpha$ is uniquely filled, from $\nu$ we can determine the content of the foundation $\mu$ ($\lambda$, respectively) needed to create a SSYT of shape $\mathcal{S}(\mu, \alpha ;k)$ ($\mathcal{S}(\lambda, \alpha ;k)$, resp.) with lattice reading word and content $\nu$.
From the content of the foundation we can determine $q$, the number of $1$'s in the foundation, and the values $a'_j, a'_{j+1},\ldots$ greater than $|\alpha|+1$.
Since Lemma~\ref{kfatfirstrowlemma} shows that the entries $\neq 1$ in the first row of the foundation strictly increase, any value $\neq 1$ that appears twice in the foundation must appear in both the first row of $\mu$ ($\lambda$, resp.) and first column of $\mu$ ($\lambda$, resp.).
These values give the set $R$.
Then $A \cup A'$ is the set of values in the foundation that are $>1$, $\leq |\alpha|+1$, but are not in $R$.
The first row of $\mu$ (first row of $\lambda$, resp.) must contain $q$ $1$'s and the values in $R$. 
After we determine the remaining entries of the first row of $\mu$ (first row of $\lambda$, resp.), the rest of the foundation is uniquely determined.

Now, to actually form a SSYT of shape $\mathcal{S}(\mu, \alpha ;k)$ ($\mathcal{S}(\lambda, \alpha ;k)$, resp.) with lattice reading word and content $\nu$, 
we only need to choose the remaining $\mu_a-q - |R|$ ($\lambda_a-q - |R|$, resp.) values from the set $(A \cup A') \cap R_{\alpha,k}$ 
to place in the first row of $\mu$ (first row of $\lambda$, resp.).
Therefore the number of SSYTx of shape $\mathcal{S}(\mu, \alpha ;k)= \kappa' / \rho'$ with lattice reading word and content $\nu$ is given by
\[ c_{\rho' \nu}^{\kappa'} = \left( \begin{array}{c}
                  |(A \cup A')\cap R_{\alpha,k}| \\
                    \mu_a -q-|R|\\
           \end{array} \right) \]
and the number of SSYTx of shape $\mathcal{S}(\lambda, \alpha ;k)= \kappa / \rho$ with lattice reading word and content $\nu$ is given by
\[ c_{\rho \nu}^{\kappa} = \left( \begin{array}{c}
                  |(A \cup A')\cap R_{\alpha,k}| \\
                    \lambda_a -q-|R|\\
           \end{array} \right).  \]

\

Since $\lambda_a \geq \lambda_l > \mu_l \geq j-1$, we have 
\begin{equation}
\label{keqneed2}
0 \geq j-1 - \lambda_a.
\end{equation}
Thus, for each $i$ we have  
\begin{eqnarray*}
\  \mu_a -q-|R|-i  & \geq & \mu_a - q- |R|-i + (j-1 - \lambda_a)
\\ & \geq &  (\mu_a -|R|) + (j-1 -|R|) - (\lambda_a -|R|) -i-q
\\ & \geq &   |A|+|A'|- (\lambda_a -|R|) -i-q
\\ & \geq &   |(A \cup A')\cap R_{\alpha,k}|- (\lambda_a -|R|) -i-q.
\end{eqnarray*}
That is, 
\begin{equation}
\label{khookynew}
\mu_a -q-|R|-i \geq  |(A \cup A')\cap R_{\alpha,k}|- (\lambda_a -|R|) -i-q,
\end{equation}
for each $i$.

\

\noindent Therefore
\begin{eqnarray*}
\  c_{\rho \nu}^{\kappa} & = & \frac{|(A \cup A')\cap R_{\alpha,k}|! }{(|(A \cup A')\cap R_{\alpha,k}|-(\lambda_a-q-|R|))!(\lambda_a-q-|R|)!}
\\ & = & \frac{|(A \cup A')\cap R_{\alpha,k}|! }{(|(A \cup A')\cap R_{\alpha,k}|-(\mu_a-q-|R|))!(\mu_a-q-|R|)!} 
\\ & & \times \prod_{i=0}^{\mu_a - \lambda_a-1} \frac{\mu_a -q- |R| -i}{|(A \cup A')\cap R_{\alpha,k}| -(\lambda_a - |R|) -i-q}
\\ & = & c_{\rho' \nu}^{\kappa'} \times \prod_{i=0}^{\mu_a - \lambda_a-1} \frac{\mu_a -q- |R| -i}{|(A \cup A')\cap R_{\alpha,k}| -(\lambda_a - |R|) -i-q}
\\ & \geq & c_{\rho' \nu}^{\kappa'} ,
\end{eqnarray*}
where we have used Equation~\ref{khookynew} in the final step. 
Since this inequality holds for all contents $\nu$ for which there was a SSYT of shape $\mathcal{S}(\mu,\alpha;k)$ with lattice reading word and content $\nu$, we have $s_{\mathcal{S}(\lambda,\alpha ;k)}-s_{\mathcal{S}(\mu,\alpha ;k)} \geq_s 0$. \qed

\end{proof}

\begin{theorem}
\label{kchainhooksb}
Let $\lambda$ and $\mu$ be hooks with $|\lambda|=|\mu|=h\leq n+k$ and $\longlroof \frac{h}{2} \longrroof \leq \lambda_a < \mu_a$, and let $0 \leq k \leq h$.
If $\mathcal{S}(\lambda, \alpha;k) \succeq_s \mathcal{S}(\mu, \alpha;k)$, then $\lambda_a \geq \mu_l+k-1$.
\end{theorem}

\begin{proof}
Since $\longlroof \frac{h}{2} \longrroof \leq \lambda_a < \mu_a$, where $|\lambda| = |\mu| =h$, we have
\[ \mu_l < \lambda_l \leq \longlroof \frac{h}{2} \longrroof \leq \lambda_a < \mu_a.\]

In the case $k=0$ we only need to show that $\lambda_a \geq \mu_l-1$, which is true since $\lambda_a \geq \longlroof \frac{h}{2} \longrroof$ and $\mu_l \leq \longlroof \frac{h}{2} \longrroof$.
\

We turn to the case $1 \leq k \leq h$. 
Towards a contradiction, suppose $\lambda_a < k+\mu_l-1$.

As before, we let $r_1=r_2=\ldots =r_k=1$ and $r_{k+1} < r_{k+2} < \ldots < r_{k+n}$ be the values of $R_{\alpha,k}$ greater than $1$.
We can create a SSYT of shape $\mu$ by filling the boxes of $\mu$ as follows.

\begin{center}
\begin{tabular}{cccccccccc}
$r_1$ & $r_2$ & $\cdots$ & $r_k$ & $r_{k+1}$  & $\cdots$ & $r_{\mu_a}$ \\
$r_{\mu_a+1}$\\
$r_{\mu_a+2}$\\
$\vdots$\\
$r_{h}$ \\
\end{tabular}
\end{center}

\noindent Since we are using $k$ 1's followed by distinct values of $R_{\alpha,k}$, it is easy to check that the resulting tableau of shape $\mu \oplus \Delta_\alpha$ has a lattice reading word. 
Thus Lemma~\ref{kfatjoinlemma} provides us with a SSYT $\mathcal{T}$ of shape $\mathcal{S}(\mu, \alpha;k)$ with lattice reading word, where $\mu$ is filled as shown above.

We now wish to count all SSYTx of shape $\mathcal{S}(\mu, \alpha;k)$ (shape $\mathcal{S}(\lambda, \alpha;k)$, respectively) with content $\nu =c(\mathcal{T})$.
Since $\Delta_\alpha$ has a unique way of being filled, we must find all semistandard fillings of $\mu$ ($\lambda$, resp.) with the values $r_1,r_2,\ldots,r_h$.
Since $r_1 = r_2 = \ldots =r_k=1$, the values $r_1$, $r_2$, $\ldots$, $r_k$ must appear in the first $k$ positions of the first row of $\mu$ ($\lambda$, resp.).
Further, once we choose $\mu_l-1$ ($\lambda_a -k$, resp.) of the values $r_{k+1} <r_{k+2} <\ldots <r_h$ to appear in the first column of $\mu$ (first row of $\lambda$, resp.), then the remaining $r$'s must appear in the first row of $\mu$ (first column of $\lambda$, resp.) and the order of all these values is uniquely determined by the semistandard conditions.

Therefore the number of SSYTx of shape $\mathcal{S}(\mu, \alpha ;k)= \kappa' / \rho'$ with lattice reading word and content $\nu =c(\mathcal{T})$ is given by
\[ c_{\rho' \nu}^{\kappa'} = \left( \begin{array}{c}
                  h-k \\
                    \mu_l -1\\
           \end{array} \right) \]
and the number of SSYTx of shape $\mathcal{S}(\lambda, \alpha ;k)= \kappa / \rho$ with lattice reading word and content $\nu=c(\mathcal{T})$ is given by
\[ c_{\rho \nu}^{\kappa} = \left( \begin{array}{c}
                  h-k \\
                    \lambda_a -k\\
           \end{array} \right).  \]

\

Since $\mu_l < \lambda_l \leq \longlroof \frac{h}{2} \longrroof \leq \lambda_a < \mu_a$, where $|\lambda|=h$, we have
\[ \mu_l + \lambda_a < \lambda_l + \lambda_a =h+1 \]
This gives $h - \lambda_a > \mu_l-1$ and we obtain
\begin{equation}
\label{antieqkkkq}
h - \lambda_a -i> \mu_l-1-i,
\end{equation}
for each $i$.

\

\noindent Therefore
\begin{eqnarray*}
\  c_{\rho' \nu}^{\kappa'} & = & \frac{(h-k)! }{(h-k-\mu_l+1)!(\mu_l-1)!}
\\ & = & \frac{(h-k)! }{(h-\lambda_a)!(\lambda_a-k)!} \times \prod_{i=0}^{\mu_l - \lambda_a+k-2} \frac{h-\lambda_a-i}{\mu_l-1 -i}
\\ & = & c_{\rho \nu}^{\kappa} \times \prod_{i=0}^{\mu_l - \lambda_a+k-2} \frac{h-\lambda_a-i}{\mu_l-1 -i}
\\ & > & c_{\rho \nu}^{\kappa} ,
\end{eqnarray*}
where we have used Equation~\ref{antieqkkkq} in the final step. Therefore $s_{\mathcal{S}(\lambda,\alpha ;k)}-s_{\mathcal{S}(\mu,\alpha ;k)} \not\geq_s 0$, which is a contradiction.
Therefore we have $\lambda_a \geq \mu_l+k-1$. \qed

\end{proof}

Finally, we turn to the hooks 
$\lambda,\mu$ satisfying $\lambda_a, \mu_l \leq \longlroof \frac{h}{2} \longrroof$.

\begin{theorem}
\label{kcrosshooks1}
Let $\lambda$ and $\mu$ be hooks with $|\lambda|=|\mu|=h \leq n+k$ and $\lambda_a, \mu_l \leq \longlroof \frac{h}{2} \longrroof$. 
If $1 \leq k \leq h$ and $\lambda_a \geq \mu_l+k-1$ then $\mathcal{S}(\lambda, \alpha;k) \succeq_s \mathcal{S}(\mu, \alpha;k)$. 
If $k=0$ and $\lambda_a \geq \mu_l$ then $\mathcal{S}(\lambda, \alpha;k) \succeq_s \mathcal{S}(\mu, \alpha;k)$.

\end{theorem}

\begin{proof}

For $1 \leq k \leq h$ we are given that $\lambda_a \geq \mu_l+k-1$ and we wish to show that $\mathcal{S}(\lambda, \alpha;k) \succeq_s \mathcal{S}(\mu, \alpha;k)$.
We claim that proof of Theorem~\ref{kchainhooks}, with a few equations verified under the current hypotheses, also proves this theorem.

Given the SSYT of shape $\mathcal{S}(\mu, \alpha;k)$ with lattice reading word and content $\nu$, in order to create the SSYT of shape $\mathcal{S}(\lambda, \alpha;k)$ with lattice reading word and content $\nu$ we first needed to check that Equation~\ref{kneedeqn3}, Equation~\ref{keqneed1} and Equation~\ref{kneed4} held.
Namely, we required that $\mu_l-j+1 \leq \lambda_l$, $|R| \leq \lambda_a-q$, and $\mu_a - |R|-q \geq \lambda_a -|R|-q$.
Since the first two equations were satisfied, we were able to fit the required values into the first row and first column of $\lambda$. Further, since $\mu_a -|R|-q \geq \lambda_a -|R|-q$, there were enough values to fill the first row of $\lambda$, and therefore we could construct the tableau with all the desired properties.

In order to show Equation~\ref{kneedeqn3} holds for the assumptions of this theorem, we note that $\mu_l \leq \longlroof \frac{h}{2} \longrroof \leq \lambda_l$.
In order to show Equation~\ref{keqneed1} holds for the assumptions of this theorem, we note that 
\[|R|\leq j-2 \leq \mu_l -1 \leq \lambda_a - k \leq \lambda_a -q.\] 
In order to show Equation~\ref{kneed4} holds for the assumptions of this theorem, we note that $\mu_a \geq \longlroof \frac{h}{2} \longrroof \geq \lambda_a$.
Therefore we can create a SSYT of shape $\mathcal{S}(\lambda, \alpha;k)$ with lattice reading word and content $\nu$ whenever there exists a SSYT of shape $\mathcal{S}(\mu, \alpha;k)$ with lattice reading word and content $\nu$.

Next, the proof of Theorem~\ref{kchainhooks} checked that, for each of these contents $\nu$, the number of SSYTx of shape $\mathcal{S}(\lambda, \alpha;k)$ with lattice reading word and content $\nu$ is greater than or equal to the number of SSYTx of shape $\mathcal{S}(\mu, \alpha;k)$ with lattice reading word and content $\nu$.
To prove this, we first required that Equation~\ref{keqneed2} held.
Namely, we required that $\lambda_a \geq j-1$.
We used this equation to show that Equation~\ref{khookynew} held for each $i$, which gave us the desired inequality for the Littlewood-Richardson numbers.
In order to show Equation~\ref{keqneed2} holds for the assumptions of this theorem, we note that 
\[ \lambda_a \geq \mu_l +k-1 \geq j-1 +k-1 \geq j-1. \]
Therefore the inequality for the Littlewood-Richardson numbers holds here as well, which proves that \[s_{\mathcal{S}(\lambda,\alpha;k)}-s_{\mathcal{S}(\mu,\alpha;k)} \geq_s 0. \]

\

In the case of $k=0$ we are assuming that $\lambda_a \geq \mu_l$. 
The only parts of the above argument that need to be adjusted are the proofs that Equation~\ref{keqneed1} and Equation~\ref{keqneed2} hold under the current hypotheses. 
For Equation~\ref{keqneed1} we note that
\[|R|\leq j-2 \leq \mu_l -1 \leq \lambda_a - 1 \leq \lambda_a  = \lambda_a -q\] 
since $q=0$, and for Equation~\ref{keqneed2} we note that 
\[  \lambda_a \geq \mu_l  \geq j-1. \] 
As before, we therefore obtain \[s_{\mathcal{S}(\lambda,\alpha;k)}-s_{\mathcal{S}(\mu,\alpha;k)} \geq_s 0. \qed\]

\end{proof}

\begin{theorem}
\label{kcrosshooks11}
Let $\lambda$ and $\mu$ be hooks with $|\lambda|=|\mu|=h\leq n+k$ and $\lambda_a, \mu_l \leq \longlroof \frac{h}{2} \longrroof$.
If $1 \leq k \leq h$ and $\mathcal{S}(\lambda, \alpha;k) \succeq_s \mathcal{S}(\mu, \alpha;k)$, then $\lambda_a \geq k+\mu_l-1$.
If $k=0$ and $\mathcal{S}(\lambda, \alpha;k) \succeq_s \mathcal{S}(\mu, \alpha;k)$, then $\lambda_a \geq \mu_l$.
\end{theorem}

\begin{proof}
We begin with the case $1 \leq k \leq h$. 
Towards a contradiction, suppose $\lambda_a < k+\mu_l-1$.

As before, we let $r_1=r_2=\ldots =r_k=1$ and $r_{k+1} < r_{k+2} < \ldots < r_{k+n}$ be the values of $R_{\alpha,k}$ greater than $1$.
We can create a SSYT of shape $\mu$ by filling the boxes of $\mu$ as follows.

\begin{center}
\begin{tabular}{cccccccccc}
$r_1$ & $r_2$ & $\cdots$ & $r_k$ & $r_{k+1}$  & $\cdots$ & $r_{\mu_a}$ \\
$r_{\mu_a+1}$\\
$r_{\mu_a+2}$\\
$\vdots$\\
$r_{h}$ \\
\end{tabular}
\end{center}

\noindent Since we are using $k$ 1's followed by distinct values of $R_{\alpha,k}$, it is easy to check that the resulting tableau of shape $\mu \oplus \Delta_\alpha$ has a lattice reading word. 
Thus Lemma~\ref{kfatjoinlemma} provides us with a SSYT $\mathcal{T}$ of shape $\mathcal{S}(\mu, \alpha;k)$ with lattice reading word, where $\mu$ is filled as shown above.

We now wish to count all SSYTx of shape $\mathcal{S}(\mu, \alpha;k)$ (shape $\mathcal{S}(\lambda, \alpha;k)$, respectively) with content $\nu =c(\mathcal{T})$.
Since $\Delta_\alpha$ has a unique way of being filled, we must find all semistandard fillings of $\mu$ ($\lambda$, resp.) with the values $r_1,r_2,\ldots,r_h$.
Since $r_1 = r_2 = \ldots =r_k=1$, the values $r_1$, $r_1$, $\ldots$ ,$r_k$ must appear in the first $k$ positions of the first row of $\mu$ ($\lambda$, resp.).
Further, once we choose $\mu_l-1$ ($\lambda_a -k$, resp.) of the values $r_{k+1} <r_{k+2} <\ldots <r_h$ to appear in the first column of $\mu$ (first row of $\lambda$, resp.), then the remaining $r$'s must appear in the first row of $\mu$ (first column of $\lambda$, resp.) and the order of all these values is uniquely determined by the semistandard conditions.

Therefore the number of SSYTx of shape $\mathcal{S}(\mu, \alpha ;k)= \kappa' / \rho'$ with lattice reading word and content $\nu =c(\mathcal{T})$ is given by
\[ c_{\rho' \nu}^{\kappa'} = \left( \begin{array}{c}
                  h-k \\
                    \mu_l -1\\
           \end{array} \right) \]
and the number of SSYTx of shape $\mathcal{S}(\lambda, \alpha ;k)= \kappa / \rho$ with lattice reading word and content $\nu=c(\mathcal{T})$ is given by
\[ c_{\rho \nu}^{\kappa} = \left( \begin{array}{c}
                  h-k \\
                    \lambda_a -k\\
           \end{array} \right).  \]

\

Since $\lambda_a, \mu_l \leq \longlroof \frac{h}{2} \longrroof$, we have $h+1 \geq \lambda_a +\mu_l$.
If we have $h+1 = \lambda_a +\mu_l$, then this implies that $h$ is odd and $\lambda_a = \mu_l = \longlroof \frac{h}{2} \longrroof = \frac{h+1}{2}$. 
Using the fact that $|\lambda| = |\mu| = h$ implies $\lambda_l = \mu_a = \longlroof \frac{h}{2} \longrroof = \frac{h+1}{2}$ as well. 
Therefore $\lambda = \mu$. 
However, we are only interested in distinct hooks $\lambda$ and $\mu$. 
Thus, among distinct pairs $\lambda$ and $\mu$, we cannot have $h+1 = \lambda_a +\mu_l$.

Thus for $\lambda \neq \mu$ with $\lambda_a, \mu_l \leq \longlroof \frac{h}{2} \longrroof$, we have $h+1 > \lambda_a +\mu_l$.
This gives $h - \lambda_a > \mu_l-1$ and we obtain
\begin{equation}
\label{antieqkkkt}
h - \lambda_a -i> \mu_l-1-i,
\end{equation}
for each $i$.

\

\noindent Therefore
\begin{eqnarray*}
\  c_{\rho' \nu}^{\kappa'} & = & \frac{(h-k)! }{(h-k-\mu_l+1)!(\mu_l-1)!}
\\ & = & \frac{(h-k)! }{(h-\lambda_a)!(\lambda_a-k)!} \times \prod_{i=0}^{\mu_l - \lambda_a+k-2} \frac{h-\lambda_a-i}{\mu_l-1 -i}
\\ & = & c_{\rho \nu}^{\kappa} \times \prod_{i=0}^{\mu_l - \lambda_a+k-2} \frac{h-\lambda_a-i}{\mu_l-1 -i}
\\ & > & c_{\rho \nu}^{\kappa} ,
\end{eqnarray*}
where we have used Equation~\ref{antieqkkkt} in the final step. Therefore $s_{\mathcal{S}(\lambda,\alpha ;k)}-s_{\mathcal{S}(\mu,\alpha ;k)} \not\geq_s 0$, which is a contradiction.
Therefore we have $\lambda_a \geq \mu_l+k-1$. 

\

Now consider the case $k=0$. 
Towards a contradiction, suppose $\lambda_a < \mu_l$.
We now let $r_{1} < r_{2} < \ldots < r_{n}$ be the values of $R_{\alpha,k}$.
We can create a SSYT of shape $\mu$ by filling the boxes of $\mu$ as follows.

\begin{center}
\begin{tabular}{cccccccccc}
$r_1$ & $r_2$ & $\cdots$ &  $r_{\mu_a}$ \\
$r_{\mu_a+1}$\\
$r_{\mu_a+2}$\\
$\vdots$\\
$r_{h}$ \\
\end{tabular}
\end{center}

As before, Lemma~\ref{kfatjoinlemma} provides us with a SSYT $\mathcal{T}$ of shape $\mathcal{S}(\mu, \alpha;k)$ with lattice reading word, where $\mu$ is filled as shown above.
We now wish to count all SSYTx of shape $\mathcal{S}(\mu, \alpha;k)$ (shape $\mathcal{S}(\lambda, \alpha;k)$, respectively) with content $\nu =c(\mathcal{T})$.
In this case only $r_1$ is required to appear at the beginning of the first row of $\mu$ ($\lambda$, resp.). 
Further, once we choose $\mu_l-1$ ($\lambda_a -1$, resp.) of the values $r_{2} < \ldots <r_h$ to appear in the first column of $\mu$ (first row of $\lambda$, resp.), then the remaining $r$'s must appear in the first row of $\mu$ (first column of $\lambda$, resp.) and the order of all these values is uniquely determined by the semistandard conditions.

Therefore the number of SSYTx of shape $\mathcal{S}(\mu, \alpha ;k)= \kappa' / \rho'$ with lattice reading word and content $\nu =c(\mathcal{T})$ is given by
\[ c_{\rho' \nu}^{\kappa'} = \left( \begin{array}{c}
                  h-1 \\
                    \mu_l -1\\
           \end{array} \right) \]
and the number of SSYTx of shape $\mathcal{S}(\lambda, \alpha ;k)= \kappa / \rho$ with lattice reading word and content $\nu=c(\mathcal{T})$ is given by
\[ c_{\rho \nu}^{\kappa} = \left( \begin{array}{c}
                  h-1 \\
                    \lambda_a -1\\
           \end{array} \right).  \]

\

As before, for $\lambda \neq \mu$ with $\lambda_a, \mu_l \leq \longlroof \frac{h}{2} \longrroof$, we have $h+1 > \lambda_a +\mu_l$.
This gives $h - \lambda_a > \mu_l-1$ and we obtain
\begin{equation}
\label{antieqkkk}
h - \lambda_a -i> \mu_l-1-i,
\end{equation}
for each $i$.

\

\noindent Therefore
\begin{eqnarray*}
\  c_{\rho' \nu}^{\kappa'} & = & \frac{(h-1)! }{(h-\mu_l)!(\mu_l-1)!}
\\ & = & \frac{(h-1)! }{(h-\lambda_a)!(\lambda_a-1)!} \times \prod_{i=0}^{\mu_l - \lambda_a -1} \frac{h-\lambda_a-i}{\mu_l-1 -i}
\\ & = & c_{\rho \nu}^{\kappa} \times \prod_{i=0}^{\mu_l - \lambda_a-1} \frac{h-\lambda_a-i}{\mu_l-1 -i}
\\ & > & c_{\rho \nu}^{\kappa} ,
\end{eqnarray*}
where we have used Equation~\ref{antieqkkk} in the final step. Therefore $s_{\mathcal{S}(\lambda,\alpha ;k)}-s_{\mathcal{S}(\mu,\alpha ;k)} \not\geq_s 0$, which is a contradiction.
Therefore we have $\lambda_a \geq \mu_l$. \qed

\end{proof}

\section{Fat Staircases with Hook Complement Foundations}

In this section we show how to extend the results of Section 3.1 to fat staircases with bad foundations where the foundations are the complements of hook diagrams. 
As before, when considering a composition $\alpha$ we shall let $n= l(\alpha)$. 
That is, $\alpha = (\alpha_1, \alpha_2, \ldots, \alpha_n)$. 
With this convention, the diagram $\delta_\alpha$ ($\Delta_\alpha$, respectively) has width $n$ and length $|\alpha|= \sum_{i=1}^{n} \alpha_i$. 
However, in this section we shall restrict the values of $k$ so that $0 \leq k \leq 1$.

Given a partition $\rho$ contained in the $a \times b$ rectangle $(b^a)$ we define the \textit{complementary partition $\rho^c$ in the rectangle $(b^a)$} by  $\rho^c = ((b^a) / \rho)^\circ$. 
That is, $\rho^c$ is the complement of $\rho$ in $(b^a)$ rotated by $180^\circ$. 
It is easy to see that this definition does define a partition. 
We display the relevant diagrams below for clarity.

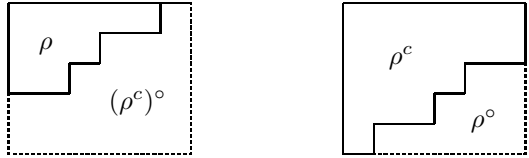
\setlength{\unitlength}{0.4mm}

\begin{picture}(100,60)(120,-25)

\put(170,5){$\rho$}
\put(193,-15){$(\rho^c)^\circ $}

\put(160,20){\line(1,0){60}}

\put(160,-10){\line(1,0){20}}
\put(180,0){\line(1,0){10}}
\put(190,10){\line(1,0){20}}

\put(160,-10){\line(0,1){30}}
\put(180,-10){\line(0,1){10}}
\put(190,0){\line(0,1){10}}
\put(210,10){\line(0,1){10}}

\put(160,-30){\dashbox{1}(60,50)[tl]{ }}

\put(285,0){$\rho^c$}
\put(312,-21){$\rho^\circ $}

\put(270,20){\line(1,0){60}}

\put(270,-30){\line(1,0){10}}
\put(280,-20){\line(1,0){20}}
\put(300,-10){\line(1,0){10}}
\put(310,0){\line(1,0){20}}

\put(270,-30){\line(0,1){50}}
\put(280,-30){\line(0,1){10}}
\put(300,-20){\line(0,1){10}}
\put(310,-10){\line(0,1){10}}
\put(330,0){\line(0,1){20}}

\put(270,-30){\dashbox{1}(60,50)[tl]{ }}

\end{picture}

\

\

For what follows, we shall make use of the following fact.

\begin{theorem} (\cite{complementcite})
\label{firstcut}
Let $\rho$ be a partition contained in the $a \times b$ rectangle $(b^a)$, $\kappa \subset \rho$ be a second partition. Then the skew diagram $\rho / \kappa$ satisfies 
\[ s_{\rho / \kappa} = \sum_{\nu \subseteq (b^a)} c_{\kappa \rho^c}^{\nu} s_{\nu^c}, \] where $c_{\kappa \rho^c}^{\nu}$ are the Littlewood-Richardson coefficients.
\end{theorem}

Given a symmetric function $f=\sum_{\nu } a_{\nu} s_{\nu}$ we define the \textit{truncated complement of $f$ in the rectangle $(b^a)$} as 
\begin{equation}
\label{cf}
c(f) = \sum_{\nu \subseteq (b^a)} a_{\nu} s_{\nu^c}.
\end{equation}
The rectangle being used should be clear from the context if it is not specifically mentioned.

\

We may now restate Theorem~\ref{firstcut} as follows.

\begin{corollary}
\label{rectcor}
Let $\rho$ be a partition contained in the $a \times b$ rectangle $(b^a)$, and $\kappa \subset \rho$ be a second partition. Then the skew diagram $\rho / \kappa$ satisfies 
\[ s_{\rho / \kappa} = c(s_\kappa s_{\rho^c}). \]
\end{corollary}

\begin{proof}
From the definition of the Littlewood-Richardson numbers, we have
\[ s_{\kappa} s_{\rho^c} =  \sum_{\nu} c_{\kappa \rho^c}^{\nu} s_{\nu}.\]
Hence \[c(s_{\kappa} s_{\rho^c}) = c(\sum_{\nu} c_{\kappa \rho^c}^{\nu} s_{\nu}) = \sum_{\nu \subseteq (b^a)} c_{\kappa \rho^c}^{\nu} s_{\nu^c}. \]
By Theorem~\ref{firstcut}, this is just $s_{\rho / \kappa}$, so we are done.\qed
\end{proof}

We begin by applying this truncation result to the shapes of the form $\rho / \kappa = \mathcal{S}(\lambda, \alpha;k)$, for $0 \leq k \leq 1$. 
In the next proof we use the operator $[s_\lambda]$ to extract the coefficient of $s_\lambda$ in an expression. That is, if $f=\sum_{\lambda} a_\lambda s_\lambda$ then $[s_\lambda] (f) = a_\lambda$.

\begin{lemma}
\label{cfatstairlemma}
For a partition $\lambda$, composition $\alpha$, and $0 \leq k \leq 1$ we have
\[c(s_{\lambda} s_{\Delta_{\alpha}}) = c(s_{\mathcal{S}(\lambda,\alpha;k)} ),\]
for any complementation in a rectangle of width $w=n+k$.
\end{lemma}

\begin{proof} 

Let the rectangle be $(w^l)$, say.
We begin by comparing $c(s_{\lambda \oplus \Delta_{\alpha}})$ and $c(s_{\mathcal{S}(\lambda,\alpha; k)})$.

Consider a content $\nu$ that contributes to $c(s_{\lambda \oplus \Delta_{\alpha}})$. 
Then $\nu \subseteq (w^l)$ and there is a SSYT $\mathcal{T}$ of shape ${\lambda \oplus \Delta_{\alpha}}$ and content $\nu^c$ with lattice reading word. 
Since $\nu^c$ is contained in a rectangle of width $w$, $\mathcal{T}$ contains at most $w=n+k$ $1$'s.
We know that exactly $n$ $1$'s appear in the copy of $\Delta_{\alpha}$. 
Thus the copy of $\lambda$ contains at most $k$ $1$'s. 
Therefore, by Lemma~\ref{kfatjoinlemma}, we can obtain a SSYT $\mathcal{T}'$ of shape $\mathcal{S}(\lambda,\alpha;k)$ with lattice reading word of content $\nu^c$ by simply filling the entries of $\lambda$ in $\mathcal{S}(\lambda,\alpha;k)$ identically to the filling of $\lambda$ in $\mathcal{T}$.
This correspondence, $\mathcal{T} \mapsto \mathcal{T}'$ gives a bijection.
That is, $[s_{\nu^c}] (s_{\lambda \oplus \Delta_{\alpha}}) = [s_{\nu^c}] (s_{\mathcal{S}(\lambda,\alpha;k)})$ for all $\nu^c \subseteq (w^l)$. 
Thus we find that $c(s_{\lambda \oplus \Delta_{\alpha}}) = c(s_{\mathcal{S}(\lambda,\alpha;k)})$.

We also have $s_{\lambda \oplus \Delta_{\alpha}} = s_{\lambda} s_{\Delta_{\alpha}}$ by Theorem~\ref{disjprod}, and so $c(s_{\lambda \oplus \Delta_{\alpha}}) = c(s_{\lambda} s_{\Delta_{\alpha}})$. 
Thus we have $c(s_{\lambda}s_{\Delta_{\alpha}} ) = c(s_{\mathcal{S}(\lambda,\alpha;k)})$, as desired. \qed

\end{proof}

\

Given a composition $\alpha = (\alpha_1,\alpha_2, \ldots, \alpha_n)$ and a width $w=n+k$, where $0 \leq k \leq 1$, we let 
\[ \alpha^r =  \left\{ \begin{array}{lcc}
            (\alpha_n, \alpha_{n-1},\ldots,\alpha_2, \alpha_1)  & \textrm{ if } & k=1 \\
            (\alpha_{n-1}, \alpha_{n-2},\ldots,\alpha_2, \alpha_1)  & \textrm{ if } & k=0 \\
         \end{array} \right. \] denote the \textit{reverse composition}.
With this definition we have ${\delta_\alpha}^c = \delta_{\alpha^r}$, where the complement is performed in the rectangle $(w^{|\alpha|})$.
We illustrate the two cases $k=0$ and $k=1$ below.

\
\

\

\

\

\

\

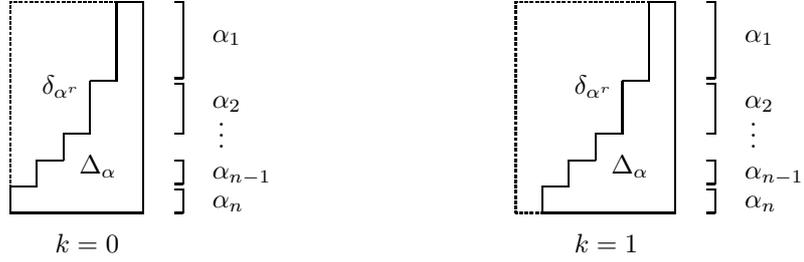
\setlength{\unitlength}{0.35mm}

\begin{picture}(100,60)(50,-15)

\put(282,05){$k=1$}
\put(87,05){$k=0$}

\put(282,65){$\delta_{\alpha^r}$}

\put(296,35){$\Delta_{\alpha}$}

\put(346,85){$\alpha_1$}
\put(346,60){$\alpha_2$}
\put(348,45){$\vdots$}
\put(148,45){$\vdots$}

\put(346,33){$\alpha_{n-1}$}
\put(346,22){$\alpha_n$}

\put(335,20){\line(0,1){9}}
\put(335,20){\line(-1,0){3}}
\put(335,29){\line(-1,0){3}}

\put(335,31){\line(0,1){9}}
\put(335,31){\line(-1,0){3}}
\put(335,40){\line(-1,0){3}}

\put(335,50){\line(0,1){19}}
\put(335,50){\line(-1,0){3}}
\put(335,69){\line(-1,0){3}}

\put(335,71){\line(0,1){29}}
\put(335,71){\line(-1,0){3}}
\put(335,100){\line(-1,0){3}}

\put(270,20){\line(1,0){50}}
\put(270,30){\line(1,0){10}}
\put(280,40){\line(1,0){10}}
\put(290,50){\line(1,0){10}}
\put(300,70){\line(1,0){10}}
\put(310,100){\line(1,0){10}}

\put(270,20){\line(0,1){10}}
\put(280,30){\line(0,1){10}}
\put(290,40){\line(0,1){10}}
\put(300,50){\line(0,1){20}}
\put(310,70){\line(0,1){30}}
\put(320,20){\line(0,1){80}}

\put(260,20){\dashbox{1}(60,80)[tl]{ }}

\put(82,65){$\delta_{\alpha^r}$}

\put(96,35){$\Delta_{\alpha}$}

\put(146,85){$\alpha_1$}
\put(146,60){$\alpha_2$}
\put(146,33){$\alpha_{n-1}$}
\put(146,22){$\alpha_n$}

\put(135,20){\line(0,1){9}}
\put(135,20){\line(-1,0){3}}
\put(135,29){\line(-1,0){3}}

\put(135,31){\line(0,1){9}}
\put(135,31){\line(-1,0){3}}
\put(135,40){\line(-1,0){3}}

\put(135,50){\line(0,1){19}}
\put(135,50){\line(-1,0){3}}
\put(135,69){\line(-1,0){3}}

\put(135,71){\line(0,1){29}}
\put(135,71){\line(-1,0){3}}
\put(135,100){\line(-1,0){3}}

\put(70,20){\line(1,0){50}}
\put(70,30){\line(1,0){10}}
\put(80,40){\line(1,0){10}}
\put(90,50){\line(1,0){10}}
\put(100,70){\line(1,0){10}}
\put(110,100){\line(1,0){10}}

\put(70,20){\line(0,1){10}}
\put(80,30){\line(0,1){10}}
\put(90,40){\line(0,1){10}}
\put(100,50){\line(0,1){20}}
\put(110,70){\line(0,1){30}}
\put(120,20){\line(0,1){80}}

\put(70,20){\dashbox{1}(50,80)[tl]{ }}

\end{picture}

Thus we have $l(\delta_{\alpha^r}) \leq l(\delta_{\alpha})$ and $w(\delta_{\alpha^r}) \leq w(\delta_{\alpha})$.
In particular, we have $|\alpha| = |\alpha^r| + (1-k)\alpha_n$.

\begin{theorem}
\label{fatstairtocutstair}
Let $\alpha$ be a composition, $w=n+k$ where $0 \leq k \leq 1$, $l \geq 1$, and $\rho$ be a partition with $|\alpha|+l$ parts such that $(w^{|\alpha|}) \subset \rho \subset (w^{|\alpha|+l})$, $\mu =(\rho_{|\alpha|+1}, \rho_{|\alpha|+2},\ldots,\rho_{|\alpha|+l})$, and $\lambda = \rho^c$ be the complement of $\rho$ in $(w^{|\alpha|+l})$.
Then $\rho / \delta_{\alpha^r} = \mathcal{S}(\mu,  \alpha;k)$ and 
\[s_{\mathcal{S}(\mu, \alpha;k)} =c(s_{\mathcal{S}(\lambda,{\alpha^r};k)}).\]
\end{theorem}

\begin{proof}
We are interested in the following diagrams, each contained in the rectangle $(w^{l})$. The first set of diagrams illustrates the case $k=1$ and the second set illustrates the case $k=0$.

\

\

\

\

\

\

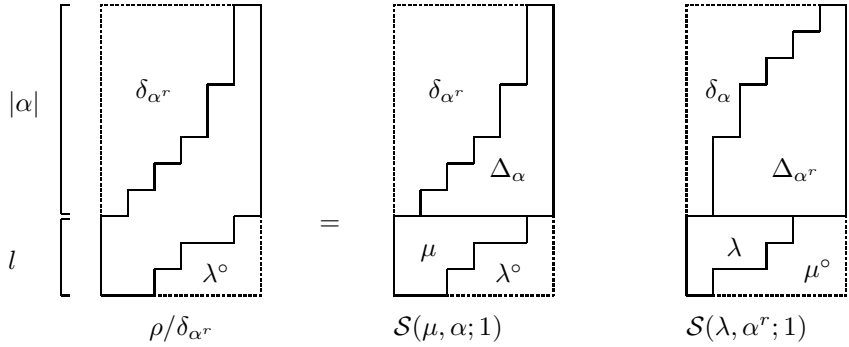
\setlength{\unitlength}{0.35mm}

\begin{picture}(100,60)(30,-5)

\put(50,20){\line(1,0){10}}
\put(60,30){\line(1,0){10}}
\put(70,40){\line(1,0){10}}
\put(80,50){\line(1,0){10}}
\put(90,70){\line(1,0){10}}
\put(100,100){\line(1,0){10}}

\put(60,20){\line(0,1){10}}
\put(70,30){\line(0,1){10}}
\put(80,40){\line(0,1){10}}
\put(90,50){\line(0,1){20}}
\put(100,70){\line(0,1){30}}
\put(110,20){\line(0,1){80}}

\put(50,-10){\line(1,0){20}}
\put(70,0){\line(1,0){10}}
\put(80,10){\line(1,0){20}}
\put(100,20){\line(1,0){10}}

\put(50,-10){\line(0,1){30}}
\put(70,-10){\line(0,1){10}}
\put(80,0){\line(0,1){10}}
\put(100,10){\line(0,1){10}}

\put(50,-10){\dashbox{1}(60,110)[tl]{ }}

\put(63,65){$\delta_{\alpha^r}$}
\put(88,-5){$\lambda^\circ$}

\put(68,-25){$\rho / \delta_{\alpha^r}$}

\put(132,15){$=$}

\put(173,65){$\delta_{\alpha^r}$}
\put(196,35){$\Delta_{\alpha}$}
\put(170,5){$\mu$}
\put(198,-5){$\lambda^\circ $}

\put(160,20){\line(1,0){60}}
\put(170,30){\line(1,0){10}}
\put(180,40){\line(1,0){10}}
\put(190,50){\line(1,0){10}}
\put(200,70){\line(1,0){10}}
\put(210,100){\line(1,0){10}}

\put(170,20){\line(0,1){10}}
\put(180,30){\line(0,1){10}}
\put(190,40){\line(0,1){10}}
\put(200,50){\line(0,1){20}}
\put(210,70){\line(0,1){30}}
\put(220,20){\line(0,1){80}}

\put(160,-10){\line(1,0){20}}
\put(180,0){\line(1,0){10}}
\put(190,10){\line(1,0){20}}

\put(160,-10){\line(0,1){30}}
\put(180,-10){\line(0,1){10}}
\put(190,0){\line(0,1){10}}
\put(210,10){\line(0,1){10}}

\put(160,-25){$\mathcal{S}(\mu,\alpha;1)$}

\put(160,-10){\dashbox{1}(60,110)[tl]{ }}

\put(277,65){$\delta_{\alpha}$}
\put(302,35){$\Delta_{\alpha^r}$}
\put(285,5){$\lambda$}
\put(314,-3){$\mu^\circ $}

\put(270,20){\line(1,0){60}}
\put(280,50){\line(1,0){10}}
\put(290,70){\line(1,0){10}}
\put(300,80){\line(1,0){10}}
\put(310,90){\line(1,0){10}}
\put(320,100){\line(1,0){10}}

\put(280,20){\line(0,1){30}}
\put(290,50){\line(0,1){20}}
\put(300,70){\line(0,1){10}}
\put(310,80){\line(0,1){10}}
\put(320,90){\line(0,1){10}}
\put(330,20){\line(0,1){80}}

\put(270,-10){\line(1,0){10}}
\put(280,0){\line(1,0){20}}
\put(300,10){\line(1,0){10}}
\put(310,20){\line(1,0){20}}

\put(270,-10){\line(0,1){30}}
\put(280,-10){\line(0,1){10}}
\put(300,0){\line(0,1){10}}
\put(310,10){\line(0,1){10}}
\put(330,20){\line(0,1){20}}

\put(270,-25){$\mathcal{S}(\lambda,{\alpha^r};1)$}

\put(270,-10){\dashbox{1}(60,110)[tl]{ }}

\put(15,60){$|\alpha|$}
\put(15,0){$l$}

\put(35,21){\line(0,1){79}}
\put(35,21){\line(1,0){3}}
\put(35,100){\line(1,0){3}}

\put(35,-10){\line(0,1){29}}
\put(35,-10){\line(1,0){3}}
\put(35,19){\line(1,0){3}}

\end{picture}

\

\

\

\

\

\

\

\

\setlength{\unitlength}{0.35mm}

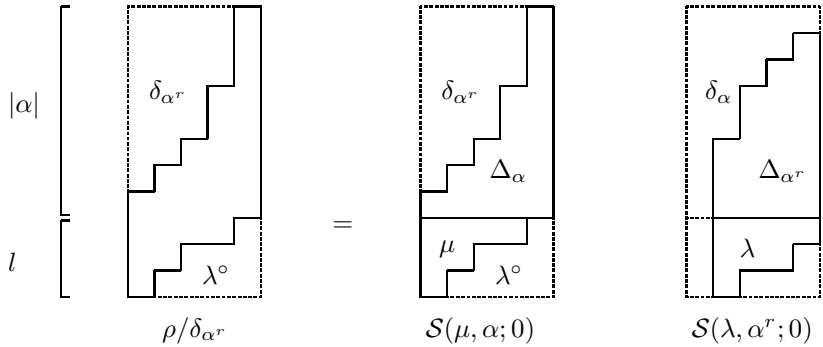
\begin{picture}(100,60)(30,-5)

\put(60,30){\line(1,0){10}}
\put(70,40){\line(1,0){10}}
\put(80,50){\line(1,0){10}}
\put(90,70){\line(1,0){10}}
\put(100,100){\line(1,0){10}}

\put(60,20){\line(0,1){10}}
\put(70,30){\line(0,1){10}}
\put(80,40){\line(0,1){10}}
\put(90,50){\line(0,1){20}}
\put(100,70){\line(0,1){30}}
\put(110,20){\line(0,1){80}}

\put(60,-10){\line(1,0){10}}
\put(70,0){\line(1,0){10}}
\put(80,10){\line(1,0){20}}
\put(100,20){\line(1,0){10}}

\put(60,-10){\line(0,1){30}}
\put(70,-10){\line(0,1){10}}
\put(80,0){\line(0,1){10}}
\put(100,10){\line(0,1){10}}

\put(60,-10){\dashbox{1}(50,110)[tl]{ }}

\put(68,65){$\delta_{\alpha^r}$}
\put(88,-5){$\lambda^\circ$}

\put(73,-25){$\rho / \delta_{\alpha^r}$}

\put(137,15){$=$}

\put(178,65){$\delta_{\alpha^r}$}
\put(196,35){$\Delta_{\alpha}$}
\put(177,7){$\mu$}
\put(198,-5){$\lambda^\circ $}

\put(170,20){\line(1,0){50}}
\put(170,30){\line(1,0){10}}
\put(180,40){\line(1,0){10}}
\put(190,50){\line(1,0){10}}
\put(200,70){\line(1,0){10}}
\put(210,100){\line(1,0){10}}

\put(170,20){\line(0,1){10}}
\put(180,30){\line(0,1){10}}
\put(190,40){\line(0,1){10}}
\put(200,50){\line(0,1){20}}
\put(210,70){\line(0,1){30}}
\put(220,20){\line(0,1){80}}

\put(170,-10){\line(1,0){10}}
\put(180,0){\line(1,0){10}}
\put(190,10){\line(1,0){20}}

\put(170,-10){\line(0,1){30}}
\put(180,-10){\line(0,1){10}}
\put(190,0){\line(0,1){10}}
\put(210,10){\line(0,1){10}}

\put(172,-25){$\mathcal{S}(\mu,\alpha;0)$}

\put(170,-10){\dashbox{1}(50,110)[tl]{ }}

\put(277,65){$\delta_{\alpha}$}
\put(297,35){$\Delta_{\alpha^r}$}
\put(290,5){$\lambda$}

\put(270,20){\dashbox{1}(10,0)[tl]{ }}

\put(280,20){\line(1,0){40}}
\put(280,50){\line(1,0){10}}
\put(290,70){\line(1,0){10}}
\put(300,80){\line(1,0){10}}
\put(310,90){\line(1,0){10}}

\put(280,20){\line(0,1){30}}
\put(290,50){\line(0,1){20}}
\put(300,70){\line(0,1){10}}
\put(310,80){\line(0,1){10}}
\put(320,20){\line(0,1){70}}

\put(280,-10){\line(1,0){10}}
\put(290,0){\line(1,0){20}}
\put(310,10){\line(1,0){10}}
\put(320,20){\line(1,0){0}}

\put(280,-10){\line(0,1){30}}
\put(290,-10){\line(0,1){10}}
\put(310,0){\line(0,1){10}}
\put(320,10){\line(0,1){10}}

\put(272,-25){$\mathcal{S}(\lambda,{\alpha^r};0)$}

\put(270,-10){\dashbox{1}(50,110)[tl]{ }}

\put(15,60){$|\alpha|$}
\put(15,0){$l$}

\put(35,21){\line(0,1){79}}
\put(35,21){\line(1,0){3}}
\put(35,100){\line(1,0){3}}

\put(35,-10){\line(0,1){29}}
\put(35,-10){\line(1,0){3}}
\put(35,19){\line(1,0){3}}

\end{picture}

\

\

\

\

It is clear from the definition of $\mu$ that we have $\mathcal{S}(\mu, \alpha;k) = \rho / \delta_{\alpha^r}$. 
Therefore we obtain
\begin{eqnarray*}
s_{\mathcal{S}(\mu, \alpha;k)} &=& s_{\rho / \delta_{\alpha^r}} \\
&=& c( s_{\lambda} s_{\delta_{\alpha^r}}) \textrm{ by Corollary~\ref{rectcor}}\\
&=& c( s_{\lambda} s_{\Delta_{\alpha^r}}) \\
&=& c(s_{\mathcal{S}(\lambda,{\alpha^r};k)} ) \textrm{ by Lemma~\ref{cfatstairlemma}},\\
\end{eqnarray*}
which is what we wanted to prove. \qed
\end{proof}

\

We now consider two hooks $\lambda$, $\mu$ both contained in a rectangle $(w^l)$ and let $\lambda^c$ and $\mu^c$ denote their complements in this rectangle. We call these \textit{hook complements}.  
For a fat staircase $\Delta_\alpha$ and $0 \leq k \leq 1$, we now inspect when the difference $s_{\mathcal{S}(\lambda^c, \alpha ;k)}-s_{\mathcal{S}(\mu^c, \alpha ;k)}$ is Schur-positive. 
Thus we are interested in the differences of skew Schur functions for pairs of diagrams such as the pair displayed below.

\

\

\

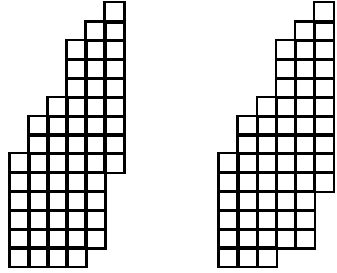
\setlength{\unitlength}{0.25mm}

\begin{picture}(00,140)(-10,-90)

\put(210,-70){\framebox(10,10)[tl]{ }}
\put(220,-70){\framebox(10,10)[tl]{ }}
\put(230,-70){\framebox(10,10)[tl]{ }}

\put(210,-60){\framebox(10,10)[tl]{ }}
\put(220,-60){\framebox(10,10)[tl]{ }}
\put(230,-60){\framebox(10,10)[tl]{ }}
\put(240,-60){\framebox(10,10)[tl]{ }}
\put(250,-60){\framebox(10,10)[tl]{ }}

\put(210,-50){\framebox(10,10)[tl]{ }}
\put(220,-50){\framebox(10,10)[tl]{ }}
\put(230,-50){\framebox(10,10)[tl]{ }}
\put(240,-50){\framebox(10,10)[tl]{ }}
\put(250,-50){\framebox(10,10)[tl]{ }}

\put(210,-40){\framebox(10,10)[tl]{ }}
\put(220,-40){\framebox(10,10)[tl]{ }}
\put(230,-40){\framebox(10,10)[tl]{ }}
\put(240,-40){\framebox(10,10)[tl]{ }}
\put(250,-40){\framebox(10,10)[tl]{ }}

\put(210,-30){\framebox(10,10)[tl]{ }}
\put(220,-30){\framebox(10,10)[tl]{ }}
\put(230,-30){\framebox(10,10)[tl]{ }}
\put(240,-30){\framebox(10,10)[tl]{ }}
\put(250,-30){\framebox(10,10)[tl]{ }}
\put(260,-30){\framebox(10,10)[tl]{ }}

\put(210,-20){\framebox(10,10)[tl]{ }}
\put(220,-20){\framebox(10,10)[tl]{ }}
\put(230,-20){\framebox(10,10)[tl]{ }}
\put(240,-20){\framebox(10,10)[tl]{ }}
\put(250,-20){\framebox(10,10)[tl]{ }}
\put(260,-20){\framebox(10,10)[tl]{ }}

\put(220,-10){\framebox(10,10)[tl]{ }}
\put(230,-10){\framebox(10,10)[tl]{ }}
\put(240,-10){\framebox(10,10)[tl]{ }}
\put(250,-10){\framebox(10,10)[tl]{ }}
\put(260,-10){\framebox(10,10)[tl]{ }}

\put(220,0){\framebox(10,10)[tl]{ }}
\put(230,0){\framebox(10,10)[tl]{ }}
\put(240,0){\framebox(10,10)[tl]{ }}
\put(250,0){\framebox(10,10)[tl]{ }}
\put(260,0){\framebox(10,10)[tl]{ }}

\put(230,10){\framebox(10,10)[tl]{ }}
\put(240,10){\framebox(10,10)[tl]{ }}
\put(250,10){\framebox(10,10)[tl]{ }}
\put(260,10){\framebox(10,10)[tl]{ }}

\put(240,20){\framebox(10,10)[tl]{ }}
\put(250,20){\framebox(10,10)[tl]{ }}
\put(260,20){\framebox(10,10)[tl]{ }}

\put(240,30){\framebox(10,10)[tl]{ }}
\put(250,30){\framebox(10,10)[tl]{ }}
\put(260,30){\framebox(10,10)[tl]{ }}

\put(240,40){\framebox(10,10)[tl]{ }}
\put(250,40){\framebox(10,10)[tl]{ }}
\put(260,40){\framebox(10,10)[tl]{ }}

\put(250,50){\framebox(10,10)[tl]{ }}
\put(260,50){\framebox(10,10)[tl]{ }}

\put(260,60){\framebox(10,10)[tl]{ }}

\put(100,-70){\framebox(10,10)[tl]{ }}
\put(110,-70){\framebox(10,10)[tl]{ }}
\put(120,-70){\framebox(10,10)[tl]{ }}
\put(130,-70){\framebox(10,10)[tl]{ }}

\put(100,-60){\framebox(10,10)[tl]{ }}
\put(110,-60){\framebox(10,10)[tl]{ }}
\put(120,-60){\framebox(10,10)[tl]{ }}
\put(130,-60){\framebox(10,10)[tl]{ }}
\put(140,-60){\framebox(10,10)[tl]{ }}

\put(100,-50){\framebox(10,10)[tl]{ }}
\put(110,-50){\framebox(10,10)[tl]{ }}
\put(120,-50){\framebox(10,10)[tl]{ }}
\put(130,-50){\framebox(10,10)[tl]{ }}
\put(140,-50){\framebox(10,10)[tl]{ }}

\put(100,-40){\framebox(10,10)[tl]{ }}
\put(110,-40){\framebox(10,10)[tl]{ }}
\put(120,-40){\framebox(10,10)[tl]{ }}
\put(130,-40){\framebox(10,10)[tl]{ }}
\put(140,-40){\framebox(10,10)[tl]{ }}

\put(100,-30){\framebox(10,10)[tl]{ }}
\put(110,-30){\framebox(10,10)[tl]{ }}
\put(120,-30){\framebox(10,10)[tl]{ }}
\put(130,-30){\framebox(10,10)[tl]{ }}
\put(140,-30){\framebox(10,10)[tl]{ }}

\put(100,-20){\framebox(10,10)[tl]{ }}
\put(110,-20){\framebox(10,10)[tl]{ }}
\put(120,-20){\framebox(10,10)[tl]{ }}
\put(130,-20){\framebox(10,10)[tl]{ }}
\put(140,-20){\framebox(10,10)[tl]{ }}
\put(150,-20){\framebox(10,10)[tl]{ }}

\put(110,-10){\framebox(10,10)[tl]{ }}
\put(120,-10){\framebox(10,10)[tl]{ }}
\put(130,-10){\framebox(10,10)[tl]{ }}
\put(140,-10){\framebox(10,10)[tl]{ }}
\put(150,-10){\framebox(10,10)[tl]{ }}

\put(110,0){\framebox(10,10)[tl]{ }}
\put(120,0){\framebox(10,10)[tl]{ }}
\put(130,0){\framebox(10,10)[tl]{ }}
\put(140,0){\framebox(10,10)[tl]{ }}
\put(150,0){\framebox(10,10)[tl]{ }}

\put(120,10){\framebox(10,10)[tl]{ }}
\put(130,10){\framebox(10,10)[tl]{ }}
\put(140,10){\framebox(10,10)[tl]{ }}
\put(150,10){\framebox(10,10)[tl]{ }}

\put(130,20){\framebox(10,10)[tl]{ }}
\put(140,20){\framebox(10,10)[tl]{ }}
\put(150,20){\framebox(10,10)[tl]{ }}

\put(130,30){\framebox(10,10)[tl]{ }}
\put(140,30){\framebox(10,10)[tl]{ }}
\put(150,30){\framebox(10,10)[tl]{ }}

\put(130,40){\framebox(10,10)[tl]{ }}
\put(140,40){\framebox(10,10)[tl]{ }}
\put(150,40){\framebox(10,10)[tl]{ }}

\put(140,50){\framebox(10,10)[tl]{ }}
\put(150,50){\framebox(10,10)[tl]{ }}

\put(150,60){\framebox(10,10)[tl]{ }}

\end{picture}
 
Our final result, Theorem~\ref{cuthookhasse} states that we obtain the same Hasse diagram for fat staircases with hook complement foundations as was obtained for fat staircases with hook foundations.
For this proof we utilise the following common notation. 
Namely, given two partitions $\lambda = (\lambda_1, \ldots, \lambda_n)$ and $\mu = (\mu_1, \ldots, \mu_m)$, we let $\lambda \cup \mu$ denote the partition that consists of the parts $\lambda_1, \ldots, \lambda_n, \mu_1, \ldots, \mu_m$ placed in weakly decreasing order. 
We shall also find it useful to treat partitions and weak compositions as vectors with non-negative integer entries that can be added componentwise.
We may add vectors of different lengths by adding zeroes to the end of the vectors.  
Further, given a positive integer $i$, we shall let $e_i$ denote the $i$-th standard basis vector.

\begin{theorem}
\label{cuthookhasse}
Let $\lambda$ and $\mu$ be hooks with $|\lambda|=|\mu|=h \leq n+k=w$ and let $0 \leq k \leq 1$.
Then $\mathcal{S}(\lambda^c,\alpha;k) \succeq_s \mathcal{S}(\mu^c,\alpha;k)$ if and only if $\mathcal{S}(\lambda,{\alpha^r};k) \succeq_s \mathcal{S}(\mu,{\alpha^r};k)$.
\end{theorem}

\begin{proof}
We wish to apply Theorem~\ref{fatstairtocutstair} to both diagrams.
To this end, we let $\rho(\lambda^c) = (w^{|\alpha|}) \cup \lambda^c$ 
and $\rho(\mu^c) = (w^{|\alpha|}) \cup \mu^c$ 
Then we have $(w^{|\alpha|}) \subset \rho(\lambda^c), \rho(\mu^c) \subset (w^{|\alpha|+l})$ so we may apply Theorem~\ref{fatstairtocutstair} to both $\rho(\lambda^c) / \delta_{\alpha^r} = \mathcal{S}(\lambda^c,\alpha;k)$ and $\rho(\mu^c) / \delta_{\alpha^r} = \mathcal{S}(\mu^c,\alpha;k)$. 
This gives
\begin{eqnarray*}
 s_{\mathcal{S}(\lambda^c,\alpha;k)} - s_{\mathcal{S}(\mu^c,\alpha;k)} 
&=&   c(s_{\mathcal{S}(\lambda,{\alpha^r};k)})-c(s_{\mathcal{S}(\mu,{\alpha^r};k)}) \\
&=&   c(s_{\mathcal{S}(\lambda,{\alpha^r};k)}-s_{\mathcal{S}(\mu,{\alpha^r};k)}),   \\
\end{eqnarray*}
where these complements are performed in the rectangle $(w^{|\alpha|+l})$. 

Now, if $\mathcal{S}(\lambda,{\alpha^r};k) \succeq_s \mathcal{S}(\mu,{\alpha^r};k)$, then the above equation shows that $\mathcal{S}(\lambda^c,\alpha;k) \succeq_s \mathcal{S}(\mu^c,\alpha;k)$ as well. 

For the converse direction, suppose that $\mathcal{S}(\lambda,{\alpha^r};k) \not\succeq_s \mathcal{S}(\mu,{\alpha^r};k)$.
Thus, by assumption, the difference $s_{\mathcal{S}(\lambda,{\alpha^r};k)}-s_{\mathcal{S}(\mu,{\alpha^r};k)}$ is not Schur-positive. 
However, we need to verify that the truncated version $c(s_{\mathcal{S}(\lambda,{\alpha^r};k)}-s_{\mathcal{S}(\mu,{\alpha^r};k)})$ is also not Schur-positive.

In Section 4, we saw that the only cases where the difference was not Schur-positive among these staircases with hook foundations were those cases covered by Theorem~\ref{kantichainhooks1}, Theorem~\ref{kantichainhooks11}, and Theorem~\ref{kcrosshooks11}. 
Thus $\lambda$ and $\mu$ must satisfy the hypotheses of one of these three theorems.
In each of these three theorems, by inspecting a particular term $s_\nu$ in the difference, it was proved that the difference was not Schur-positive. 
We need only check that for each theorem the partition $\nu$ constructed satisfies $\nu \subseteq (w^{|\alpha|+l})$. 
In this way, we prove that the term $s_\nu$ also appears in the truncated difference and hence shows that the this truncated difference is not Schur-positive.

\

In both Theorem~\ref{kantichainhooks11} and Theorem~\ref{kcrosshooks11} we used the content $\nu = \delta_{\alpha^r} + \sum_{i=1}^{h} e_{r_{i}}$, where $0 \leq k \leq 1$ implies that $r_1 < r_2 < \dots$ are the values of $R_{\alpha^r,k}$. 
We have $w(\delta_{\alpha^r}) \leq w(\delta_{\alpha}) = n$. 
Further, since the $r_i$ are distinct, adding the terms $\sum_{i=1}^{h} e_{r_{i}}$ to $\delta_{\alpha^r}$ can only increase the width by 1, and this only happens when $r_1=1$ which implies that $k=1$.
Thus, in either case, $w(\nu) \leq n +k =w$. 
Also, since $l(\delta_{\alpha^r}) \leq l(\delta_{\alpha}) = |\alpha|$ and each $r_i \leq |\alpha|+1$, we have $l(\nu) \leq |\alpha|+1 \leq |\alpha|+l$. 
Therefore $\nu$ is contained in the rectangle $(w^{|\alpha|+l})$.

In Theorem~\ref{kantichainhooks1} we used $\nu = \delta_{\alpha^r} + \sum_{i=1}^{\lambda_a-1} e_{r_{i}} + (0^{|\alpha^r|},1^{\lambda_l})$, where $r_1 < r_2 < \dots$ are values of $R_{\alpha^r,k}$. 
As in the previous case we find that $w(\nu) \leq n +k =w$. 
For the length of $\nu$ we have $l(\nu) = |\alpha^r| + \lambda_l \leq |\alpha| + l$, since $|\alpha^r| \leq |\alpha|$ and $\lambda \subseteq (w^l)$. 
Therefore $\nu$ is contained in the rectangle $(w^{|\alpha|+l})$.

\

Thus, in each case $\nu$ is contained in the rectangle $(w^{|\alpha|+l})$. Therefore the term $s_\nu$ in the difference $s_{\mathcal{S}(\lambda,{\alpha^r};k)}-s_{\mathcal{S}(\mu,{\alpha^r};k)}$, is also in the difference $c(s_{\mathcal{S}(\lambda,{\alpha^r};k)}-s_{\mathcal{S}(\mu,{\alpha^r};k)})$. 
Since this term has a negative coefficient, it shows that $c(s_{\mathcal{S}(\lambda,{\alpha^r};k)}-s_{\mathcal{S}(\mu,{\alpha^r};k)})$, and hence $s_{\mathcal{S}(\lambda^c,\alpha;k)} - s_{\mathcal{S}(\mu^c,\alpha;k)}$ is not Schur-positive.
That is, $\mathcal{S}(\lambda^c,\alpha;k) \not\succeq_s \mathcal{S}(\mu^c,\alpha;k)$.
This completes the converse direction.

Thus we have shown that $\mathcal{S}(\lambda^c,\alpha;k) \succeq_s \mathcal{S}(\mu^c,\alpha;k)$ if and only if $\mathcal{S}(\lambda,{\alpha^r};k) \succeq_s \mathcal{S}(\mu,{\alpha^r};k)$. \qed

\end{proof}

\begin{example} Here we see the Hasse diagram obtained by considering all diagrams of the form $\mathcal{S}(\lambda^c,\alpha;k)$ where $\alpha = (1,1,3,1,2)$, $k=1$, and $\lambda$ is a hook of size 6, where $\lambda^c$ is computed in the rectangle $(6^6)$.

\

\

\setlength{\unitlength}{0.2mm}

\begin{picture}(00,140)(-60,-90)

\put(42,-80){\line(1,-1){263}}
\put(130,-80){\line(1,-1){165}}
\put(130,-80){\line(2,-3){176}}
\put(240,-80){\line(1,-1){60}}
\put(240,-80){\line(1,-3){55}}
\put(240,-80){\line(1,-4){66}}

\put(345,-345){\line(0,1){30}}
\put(345,-190){\line(0,1){30}}

\put(320,-460){\framebox(10,10)[tl]{ }}
\put(330,-460){\framebox(10,10)[tl]{ }}
\put(340,-460){\framebox(10,10)[tl]{ }}
\put(350,-460){\framebox(10,10)[tl]{ }}
\put(360,-460){\framebox(10,10)[tl]{ }}
\put(370,-460){\framebox(10,10)[tl]{ }}

\put(320,-450){\framebox(10,10)[tl]{ }}
\put(330,-450){\framebox(10,10)[tl]{ }}
\put(340,-450){\framebox(10,10)[tl]{ }}
\put(350,-450){\framebox(10,10)[tl]{ }}
\put(360,-450){\framebox(10,10)[tl]{ }}
\put(370,-450){\framebox(10,10)[tl]{ }}

\put(320,-440){\framebox(10,10)[tl]{ }}
\put(330,-440){\framebox(10,10)[tl]{ }}
\put(340,-440){\framebox(10,10)[tl]{ }}
\put(350,-440){\framebox(10,10)[tl]{ }}
\put(360,-440){\framebox(10,10)[tl]{ }}
\put(370,-440){\framebox(10,10)[tl]{ }}

\put(320,-430){\framebox(10,10)[tl]{ }}
\put(330,-430){\framebox(10,10)[tl]{ }}
\put(340,-430){\framebox(10,10)[tl]{ }}
\put(350,-430){\framebox(10,10)[tl]{ }}
\put(360,-430){\framebox(10,10)[tl]{ }}
\put(370,-430){\framebox(10,10)[tl]{ }}

\put(320,-420){\framebox(10,10)[tl]{ }}
\put(330,-420){\framebox(10,10)[tl]{ }}
\put(340,-420){\framebox(10,10)[tl]{ }}
\put(350,-420){\framebox(10,10)[tl]{ }}
\put(360,-420){\framebox(10,10)[tl]{ }}
\put(370,-420){\framebox(10,10)[tl]{ }}

\put(330,-410){\framebox(10,10)[tl]{ }}
\put(340,-410){\framebox(10,10)[tl]{ }}
\put(350,-410){\framebox(10,10)[tl]{ }}
\put(360,-410){\framebox(10,10)[tl]{ }}
\put(370,-410){\framebox(10,10)[tl]{ }}

\put(330,-400){\framebox(10,10)[tl]{ }}
\put(340,-400){\framebox(10,10)[tl]{ }}
\put(350,-400){\framebox(10,10)[tl]{ }}
\put(360,-400){\framebox(10,10)[tl]{ }}
\put(370,-400){\framebox(10,10)[tl]{ }}

\put(340,-390){\framebox(10,10)[tl]{ }}
\put(350,-390){\framebox(10,10)[tl]{ }}
\put(360,-390){\framebox(10,10)[tl]{ }}
\put(370,-390){\framebox(10,10)[tl]{ }}

\put(350,-380){\framebox(10,10)[tl]{ }}
\put(360,-380){\framebox(10,10)[tl]{ }}
\put(370,-380){\framebox(10,10)[tl]{ }}

\put(350,-370){\framebox(10,10)[tl]{ }}
\put(360,-370){\framebox(10,10)[tl]{ }}
\put(370,-370){\framebox(10,10)[tl]{ }}

\put(350,-360){\framebox(10,10)[tl]{ }}
\put(360,-360){\framebox(10,10)[tl]{ }}
\put(370,-360){\framebox(10,10)[tl]{ }}

\put(360,-350){\framebox(10,10)[tl]{ }}
\put(370,-350){\framebox(10,10)[tl]{ }}

\put(370,-340){\framebox(10,10)[tl]{ }}

\put(320,-310){\framebox(10,10)[tl]{ }}

\put(320,-300){\framebox(10,10)[tl]{ }}
\put(330,-300){\framebox(10,10)[tl]{ }}
\put(340,-300){\framebox(10,10)[tl]{ }}
\put(350,-300){\framebox(10,10)[tl]{ }}
\put(360,-300){\framebox(10,10)[tl]{ }}

\put(320,-290){\framebox(10,10)[tl]{ }}
\put(330,-290){\framebox(10,10)[tl]{ }}
\put(340,-290){\framebox(10,10)[tl]{ }}
\put(350,-290){\framebox(10,10)[tl]{ }}
\put(360,-290){\framebox(10,10)[tl]{ }}
\put(370,-290){\framebox(10,10)[tl]{ }}

\put(320,-280){\framebox(10,10)[tl]{ }}
\put(330,-280){\framebox(10,10)[tl]{ }}
\put(340,-280){\framebox(10,10)[tl]{ }}
\put(350,-280){\framebox(10,10)[tl]{ }}
\put(360,-280){\framebox(10,10)[tl]{ }}
\put(370,-280){\framebox(10,10)[tl]{ }}

\put(320,-270){\framebox(10,10)[tl]{ }}
\put(330,-270){\framebox(10,10)[tl]{ }}
\put(340,-270){\framebox(10,10)[tl]{ }}
\put(350,-270){\framebox(10,10)[tl]{ }}
\put(360,-270){\framebox(10,10)[tl]{ }}
\put(370,-270){\framebox(10,10)[tl]{ }}

\put(320,-260){\framebox(10,10)[tl]{ }}
\put(330,-260){\framebox(10,10)[tl]{ }}
\put(340,-260){\framebox(10,10)[tl]{ }}
\put(350,-260){\framebox(10,10)[tl]{ }}
\put(360,-260){\framebox(10,10)[tl]{ }}
\put(370,-260){\framebox(10,10)[tl]{ }}

\put(330,-250){\framebox(10,10)[tl]{ }}
\put(340,-250){\framebox(10,10)[tl]{ }}
\put(350,-250){\framebox(10,10)[tl]{ }}
\put(360,-250){\framebox(10,10)[tl]{ }}
\put(370,-250){\framebox(10,10)[tl]{ }}

\put(330,-240){\framebox(10,10)[tl]{ }}
\put(340,-240){\framebox(10,10)[tl]{ }}
\put(350,-240){\framebox(10,10)[tl]{ }}
\put(360,-240){\framebox(10,10)[tl]{ }}
\put(370,-240){\framebox(10,10)[tl]{ }}

\put(340,-230){\framebox(10,10)[tl]{ }}
\put(350,-230){\framebox(10,10)[tl]{ }}
\put(360,-230){\framebox(10,10)[tl]{ }}
\put(370,-230){\framebox(10,10)[tl]{ }}

\put(350,-220){\framebox(10,10)[tl]{ }}
\put(360,-220){\framebox(10,10)[tl]{ }}
\put(370,-220){\framebox(10,10)[tl]{ }}

\put(350,-210){\framebox(10,10)[tl]{ }}
\put(360,-210){\framebox(10,10)[tl]{ }}
\put(370,-210){\framebox(10,10)[tl]{ }}

\put(350,-200){\framebox(10,10)[tl]{ }}
\put(360,-200){\framebox(10,10)[tl]{ }}
\put(370,-200){\framebox(10,10)[tl]{ }}

\put(360,-190){\framebox(10,10)[tl]{ }}
\put(370,-190){\framebox(10,10)[tl]{ }}

\put(370,-180){\framebox(10,10)[tl]{ }}

\put(320,-150){\framebox(10,10)[tl]{ }}
\put(330,-150){\framebox(10,10)[tl]{ }}

\put(320,-140){\framebox(10,10)[tl]{ }}
\put(330,-140){\framebox(10,10)[tl]{ }}
\put(340,-140){\framebox(10,10)[tl]{ }}
\put(350,-140){\framebox(10,10)[tl]{ }}
\put(360,-140){\framebox(10,10)[tl]{ }}

\put(320,-130){\framebox(10,10)[tl]{ }}
\put(330,-130){\framebox(10,10)[tl]{ }}
\put(340,-130){\framebox(10,10)[tl]{ }}
\put(350,-130){\framebox(10,10)[tl]{ }}
\put(360,-130){\framebox(10,10)[tl]{ }}

\put(320,-120){\framebox(10,10)[tl]{ }}
\put(330,-120){\framebox(10,10)[tl]{ }}
\put(340,-120){\framebox(10,10)[tl]{ }}
\put(350,-120){\framebox(10,10)[tl]{ }}
\put(360,-120){\framebox(10,10)[tl]{ }}
\put(370,-120){\framebox(10,10)[tl]{ }}

\put(320,-110){\framebox(10,10)[tl]{ }}
\put(330,-110){\framebox(10,10)[tl]{ }}
\put(340,-110){\framebox(10,10)[tl]{ }}
\put(350,-110){\framebox(10,10)[tl]{ }}
\put(360,-110){\framebox(10,10)[tl]{ }}
\put(370,-110){\framebox(10,10)[tl]{ }}

\put(320,-100){\framebox(10,10)[tl]{ }}
\put(330,-100){\framebox(10,10)[tl]{ }}
\put(340,-100){\framebox(10,10)[tl]{ }}
\put(350,-100){\framebox(10,10)[tl]{ }}
\put(360,-100){\framebox(10,10)[tl]{ }}
\put(370,-100){\framebox(10,10)[tl]{ }}

\put(330,-90){\framebox(10,10)[tl]{ }}
\put(340,-90){\framebox(10,10)[tl]{ }}
\put(350,-90){\framebox(10,10)[tl]{ }}
\put(360,-90){\framebox(10,10)[tl]{ }}
\put(370,-90){\framebox(10,10)[tl]{ }}

\put(330,-80){\framebox(10,10)[tl]{ }}
\put(340,-80){\framebox(10,10)[tl]{ }}
\put(350,-80){\framebox(10,10)[tl]{ }}
\put(360,-80){\framebox(10,10)[tl]{ }}
\put(370,-80){\framebox(10,10)[tl]{ }}

\put(340,-70){\framebox(10,10)[tl]{ }}
\put(350,-70){\framebox(10,10)[tl]{ }}
\put(360,-70){\framebox(10,10)[tl]{ }}
\put(370,-70){\framebox(10,10)[tl]{ }}

\put(350,-60){\framebox(10,10)[tl]{ }}
\put(360,-60){\framebox(10,10)[tl]{ }}
\put(370,-60){\framebox(10,10)[tl]{ }}

\put(350,-50){\framebox(10,10)[tl]{ }}
\put(360,-50){\framebox(10,10)[tl]{ }}
\put(370,-50){\framebox(10,10)[tl]{ }}

\put(350,-40){\framebox(10,10)[tl]{ }}
\put(360,-40){\framebox(10,10)[tl]{ }}
\put(370,-40){\framebox(10,10)[tl]{ }}

\put(360,-30){\framebox(10,10)[tl]{ }}
\put(370,-30){\framebox(10,10)[tl]{ }}

\put(370,-20){\framebox(10,10)[tl]{ }}

\put(210,-70){\framebox(10,10)[tl]{ }}
\put(220,-70){\framebox(10,10)[tl]{ }}
\put(230,-70){\framebox(10,10)[tl]{ }}

\put(210,-60){\framebox(10,10)[tl]{ }}
\put(220,-60){\framebox(10,10)[tl]{ }}
\put(230,-60){\framebox(10,10)[tl]{ }}
\put(240,-60){\framebox(10,10)[tl]{ }}
\put(250,-60){\framebox(10,10)[tl]{ }}

\put(210,-50){\framebox(10,10)[tl]{ }}
\put(220,-50){\framebox(10,10)[tl]{ }}
\put(230,-50){\framebox(10,10)[tl]{ }}
\put(240,-50){\framebox(10,10)[tl]{ }}
\put(250,-50){\framebox(10,10)[tl]{ }}

\put(210,-40){\framebox(10,10)[tl]{ }}
\put(220,-40){\framebox(10,10)[tl]{ }}
\put(230,-40){\framebox(10,10)[tl]{ }}
\put(240,-40){\framebox(10,10)[tl]{ }}
\put(250,-40){\framebox(10,10)[tl]{ }}

\put(210,-30){\framebox(10,10)[tl]{ }}
\put(220,-30){\framebox(10,10)[tl]{ }}
\put(230,-30){\framebox(10,10)[tl]{ }}
\put(240,-30){\framebox(10,10)[tl]{ }}
\put(250,-30){\framebox(10,10)[tl]{ }}
\put(260,-30){\framebox(10,10)[tl]{ }}

\put(210,-20){\framebox(10,10)[tl]{ }}
\put(220,-20){\framebox(10,10)[tl]{ }}
\put(230,-20){\framebox(10,10)[tl]{ }}
\put(240,-20){\framebox(10,10)[tl]{ }}
\put(250,-20){\framebox(10,10)[tl]{ }}
\put(260,-20){\framebox(10,10)[tl]{ }}

\put(220,-10){\framebox(10,10)[tl]{ }}
\put(230,-10){\framebox(10,10)[tl]{ }}
\put(240,-10){\framebox(10,10)[tl]{ }}
\put(250,-10){\framebox(10,10)[tl]{ }}
\put(260,-10){\framebox(10,10)[tl]{ }}

\put(220,0){\framebox(10,10)[tl]{ }}
\put(230,0){\framebox(10,10)[tl]{ }}
\put(240,0){\framebox(10,10)[tl]{ }}
\put(250,0){\framebox(10,10)[tl]{ }}
\put(260,0){\framebox(10,10)[tl]{ }}

\put(230,10){\framebox(10,10)[tl]{ }}
\put(240,10){\framebox(10,10)[tl]{ }}
\put(250,10){\framebox(10,10)[tl]{ }}
\put(260,10){\framebox(10,10)[tl]{ }}

\put(240,20){\framebox(10,10)[tl]{ }}
\put(250,20){\framebox(10,10)[tl]{ }}
\put(260,20){\framebox(10,10)[tl]{ }}

\put(240,30){\framebox(10,10)[tl]{ }}
\put(250,30){\framebox(10,10)[tl]{ }}
\put(260,30){\framebox(10,10)[tl]{ }}

\put(240,40){\framebox(10,10)[tl]{ }}
\put(250,40){\framebox(10,10)[tl]{ }}
\put(260,40){\framebox(10,10)[tl]{ }}

\put(250,50){\framebox(10,10)[tl]{ }}
\put(260,50){\framebox(10,10)[tl]{ }}

\put(260,60){\framebox(10,10)[tl]{ }}

\put(100,-70){\framebox(10,10)[tl]{ }}
\put(110,-70){\framebox(10,10)[tl]{ }}
\put(120,-70){\framebox(10,10)[tl]{ }}
\put(130,-70){\framebox(10,10)[tl]{ }}

\put(100,-60){\framebox(10,10)[tl]{ }}
\put(110,-60){\framebox(10,10)[tl]{ }}
\put(120,-60){\framebox(10,10)[tl]{ }}
\put(130,-60){\framebox(10,10)[tl]{ }}
\put(140,-60){\framebox(10,10)[tl]{ }}

\put(100,-50){\framebox(10,10)[tl]{ }}
\put(110,-50){\framebox(10,10)[tl]{ }}
\put(120,-50){\framebox(10,10)[tl]{ }}
\put(130,-50){\framebox(10,10)[tl]{ }}
\put(140,-50){\framebox(10,10)[tl]{ }}

\put(100,-40){\framebox(10,10)[tl]{ }}
\put(110,-40){\framebox(10,10)[tl]{ }}
\put(120,-40){\framebox(10,10)[tl]{ }}
\put(130,-40){\framebox(10,10)[tl]{ }}
\put(140,-40){\framebox(10,10)[tl]{ }}

\put(100,-30){\framebox(10,10)[tl]{ }}
\put(110,-30){\framebox(10,10)[tl]{ }}
\put(120,-30){\framebox(10,10)[tl]{ }}
\put(130,-30){\framebox(10,10)[tl]{ }}
\put(140,-30){\framebox(10,10)[tl]{ }}

\put(100,-20){\framebox(10,10)[tl]{ }}
\put(110,-20){\framebox(10,10)[tl]{ }}
\put(120,-20){\framebox(10,10)[tl]{ }}
\put(130,-20){\framebox(10,10)[tl]{ }}
\put(140,-20){\framebox(10,10)[tl]{ }}
\put(150,-20){\framebox(10,10)[tl]{ }}

\put(110,-10){\framebox(10,10)[tl]{ }}
\put(120,-10){\framebox(10,10)[tl]{ }}
\put(130,-10){\framebox(10,10)[tl]{ }}
\put(140,-10){\framebox(10,10)[tl]{ }}
\put(150,-10){\framebox(10,10)[tl]{ }}

\put(110,0){\framebox(10,10)[tl]{ }}
\put(120,0){\framebox(10,10)[tl]{ }}
\put(130,0){\framebox(10,10)[tl]{ }}
\put(140,0){\framebox(10,10)[tl]{ }}
\put(150,0){\framebox(10,10)[tl]{ }}

\put(120,10){\framebox(10,10)[tl]{ }}
\put(130,10){\framebox(10,10)[tl]{ }}
\put(140,10){\framebox(10,10)[tl]{ }}
\put(150,10){\framebox(10,10)[tl]{ }}

\put(130,20){\framebox(10,10)[tl]{ }}
\put(140,20){\framebox(10,10)[tl]{ }}
\put(150,20){\framebox(10,10)[tl]{ }}

\put(130,30){\framebox(10,10)[tl]{ }}
\put(140,30){\framebox(10,10)[tl]{ }}
\put(150,30){\framebox(10,10)[tl]{ }}

\put(130,40){\framebox(10,10)[tl]{ }}
\put(140,40){\framebox(10,10)[tl]{ }}
\put(150,40){\framebox(10,10)[tl]{ }}

\put(140,50){\framebox(10,10)[tl]{ }}
\put(150,50){\framebox(10,10)[tl]{ }}

\put(150,60){\framebox(10,10)[tl]{ }}

\put(-10,-70){\framebox(10,10)[tl]{ }}
\put(0,-70){\framebox(10,10)[tl]{ }}
\put(10,-70){\framebox(10,10)[tl]{ }}
\put(20,-70){\framebox(10,10)[tl]{ }}
\put(30,-70){\framebox(10,10)[tl]{ }}

\put(-10,-60){\framebox(10,10)[tl]{ }}
\put(0,-60){\framebox(10,10)[tl]{ }}
\put(10,-60){\framebox(10,10)[tl]{ }}
\put(20,-60){\framebox(10,10)[tl]{ }}
\put(30,-60){\framebox(10,10)[tl]{ }}

\put(-10,-50){\framebox(10,10)[tl]{ }}
\put(0,-50){\framebox(10,10)[tl]{ }}
\put(10,-50){\framebox(10,10)[tl]{ }}
\put(20,-50){\framebox(10,10)[tl]{ }}
\put(30,-50){\framebox(10,10)[tl]{ }}

\put(-10,-40){\framebox(10,10)[tl]{ }}
\put(0,-40){\framebox(10,10)[tl]{ }}
\put(10,-40){\framebox(10,10)[tl]{ }}
\put(20,-40){\framebox(10,10)[tl]{ }}
\put(30,-40){\framebox(10,10)[tl]{ }}

\put(-10,-30){\framebox(10,10)[tl]{ }}
\put(0,-30){\framebox(10,10)[tl]{ }}
\put(10,-30){\framebox(10,10)[tl]{ }}
\put(20,-30){\framebox(10,10)[tl]{ }}
\put(30,-30){\framebox(10,10)[tl]{ }}

\put(-10,-20){\framebox(10,10)[tl]{ }}
\put(0,-20){\framebox(10,10)[tl]{ }}
\put(10,-20){\framebox(10,10)[tl]{ }}
\put(20,-20){\framebox(10,10)[tl]{ }}
\put(30,-20){\framebox(10,10)[tl]{ }}

\put(0,-10){\framebox(10,10)[tl]{ }}
\put(10,-10){\framebox(10,10)[tl]{ }}
\put(20,-10){\framebox(10,10)[tl]{ }}
\put(30,-10){\framebox(10,10)[tl]{ }}
\put(40,-10){\framebox(10,10)[tl]{ }}

\put(0,0){\framebox(10,10)[tl]{ }}
\put(10,0){\framebox(10,10)[tl]{ }}
\put(20,0){\framebox(10,10)[tl]{ }}
\put(30,0){\framebox(10,10)[tl]{ }}
\put(40,0){\framebox(10,10)[tl]{ }}

\put(10,10){\framebox(10,10)[tl]{ }}
\put(20,10){\framebox(10,10)[tl]{ }}
\put(30,10){\framebox(10,10)[tl]{ }}
\put(40,10){\framebox(10,10)[tl]{ }}

\put(20,20){\framebox(10,10)[tl]{ }}
\put(30,20){\framebox(10,10)[tl]{ }}
\put(40,20){\framebox(10,10)[tl]{ }}

\put(20,30){\framebox(10,10)[tl]{ }}
\put(30,30){\framebox(10,10)[tl]{ }}
\put(40,30){\framebox(10,10)[tl]{ }}

\put(20,40){\framebox(10,10)[tl]{ }}
\put(30,40){\framebox(10,10)[tl]{ }}
\put(40,40){\framebox(10,10)[tl]{ }}

\put(30,50){\framebox(10,10)[tl]{ }}
\put(40,50){\framebox(10,10)[tl]{ }}

\put(40,60){\framebox(10,10)[tl]{ }}

\end{picture}

\end{example}

\

\

\

\

\

\

\

\

\

\

\

\

\

\

\

\

\end{document}